\documentclass[12pt,twoside]{amsart}
\usepackage{latexsym,amsmath,amsopn,amssymb,amsthm,amsfonts}
\usepackage[T2A]{fontenc}
\usepackage[cp1251]{inputenc}
\usepackage[russian,english]{babel}
\usepackage{graphicx}

\newtheorem{theorem}{Theorem}

\newtheorem{proposition}{Proposition}

\newtheorem{remark}{Remark}

\newtheorem{definition}{Definition}

\newtheorem{example}{Example}

\newtheorem{corollary}{Corollary}

\newcommand{\ad}{\operatorname{ad}}
\newcommand{\Ad}{\operatorname{Ad}}
\newcommand{\trace}{\operatorname{trace}}
\newcommand{\Mod}{\operatorname{mod}}
\newcommand{\Spec}{\operatorname{Spec}}
\newcommand{\Spin}{\operatorname{Spin}}
\newcommand{\Sp}{\operatorname{Sp}}
\newcommand{\SU}{\operatorname{SU}}
\newcommand{\SO}{\operatorname{SO}}
\newcommand{\PSp}{\operatorname{PSp}}
\newcommand{\PSU}{\operatorname{PSU}}

{
\newcounter{table_counter}	
\newcommand{\Table}{\refstepcounter{table_counter}\newline\newline{\it Table \arabic{table_counter}}.}

\begin{document}

\title[Laplace operator spectrum]{Laplace operator spectrum on connected compact simple rank three Lie groups}
\author{V.\,N.\,Berestovskii, I.\,A.\,Zubareva, V.\,M.\,Svirkin}
\thanks{The first author was partially supported by the Russian Foundation
for Basic Research (Grant 14-01-00068-a) and a grant of the Government
of the Russian Federation for the State Support of Scientific Research
(Agreement \No 14.B25.31.0029)}
\address{V.N.Berestovskii}
\address{Sobolev Institute of mathematics SD RAS, \newline 4 Acad. Koptyug avenue, Novosibirsk, 630090, Russia}
\email{vberestov@inbox.ru}
\address{I.A.Zubareva}
\address{Sobolev Institute of mathematics SD RAS, Omsk Branch, \newline 13 Pevtsova str.,  Omsk, 644099, Russia}
\email{i\_gribanova@mail.ru}
\address{V.M.Svirkin}
\address{Omsk F.M.~Dostoevsky State university, \newline 55а Mira avenue, Omsk, 644077, Russia}
\email{v\_svirkin@mail.ru}
\maketitle
\maketitle {\small
\begin{quote}
\noindent{\sc Abstract.}
In this paper are given explicit calculations of Laplace operator spectrum for smooth real/complex-valued functions on all connected compact simple rank three Lie groups with biinvariant Riemannian metric and established a connection of obtained formulas with the number theory and integer ternary and binary quadratic forms.
\end{quote}}

{\small
\begin{quote}
MSC 22E30, 49J15, 53C17
\end{quote}}

\section*{Introduction}

In paper \cite{BerSvir1} are studied the spectrum of the Laplace(-Beltrami) operator on smooth real-valued functions defined on compact normal homogeneous Riemannian manifolds. It was shown that in some sense this problem could be reduced to considerations of compact simply connected (connected) simple Lie groups $G$ with biinvariant (i.e. invariant relative to left and right shifts) Riemannian metric $\nu$. In the last case it is suggested an algorithm for the search of Laplacian spectrum via representations of Lie algebras of Lie groups $G$.

In addition, it is proved that elements of irreducible matrix real representations $r$ of the Lie group $G$ are eigenfunctions of the Laplacian, corresponding one and the same eigenvalue
$\lambda_{r}\leq 0$, moreover $\lambda_{r}=0$ if and only if $r(G)=\{(1)\}$. Such functions constitute a basis $F$ for all real eigenfunctions of the Laplacian.

The spectrum of the Laplace operator on $(G,\nu)$ is expressed by dimensions $d_r$ of representations $r$, dimensions $m$ of the Lie group $G$, the scalar product $\nu(e)$ on the Lie algebra $\frak{g}$ of the Lie group $G$ and associated to real representations
$\rho=dr(e)$ of the Lie algebra $\frak{g}$ a bilinear symmetric form $k_{\rho}$. Some description of all irreducible real representations $\rho$ of a real simple Lie algebra $\frak{g}$ via irreducible complex representations of complex hull $\frak{k}$ of $\frak{g}$ is obtained in book \cite{Onishik} by A.L.~Onishchik and is presented in article \cite{BerSvir1}. Any irreducible complex representation of the Lie algebra $\frak{k}$ is defined by its highest weight $\Lambda$. Its dimension could be calculated via $\Lambda$ by known formula of H.~Weyl. There exists formula
which permits to calculate $\lambda_{r}$ by corresponding highest weight $\Lambda$ and $\nu(e)$. In common these formulas give an algorithm of calculation of the spectrum of the Laplacian for real-valued functions on $(G,\nu)$.

In paper \cite{BerSvir2} is formulated an improved algorithm, by means of whom and usage of results in number theory and theory of binary quadratic forms with integer coefficients are given explicit calculations of the Laplacian spectrum  for all compact simply connected simple rank two Lie groups. The improving of algorithm consisted mainly in the passage to complex case under calculation of multiplicity of eigenvalues of the Laplacian.

In work \cite{Svir3} the search algorithm for Laplacian spectrum from \cite{BerSvir2} is generalized to non simply connected case, i.e. it is considered the case of arbitrary compact connected simple Lie group $G$. In \cite{Svir3} is calculated also Laplacian spectrum for all compact connected simple Lie groups with ranks one and two.

Lie groups with biinvariant Riemannian metric are partial cases of Riemannian symmetric spaces. In article \cite{Ber1} is given a detailed survey of known methods for calculation of Laplace-Beltrami operator spectrum for real-valued and complex-valued smooth functions on compact connected normal homogeneous Riemannian manifolds. These methods are especially effective for compact normal homogeneous Riemannian manifolds with positive Euler characteristic and compact simply connected
irreducible Riemannian symmetric spaces.

In paper \cite{Ber2} by simple method, using Hopf fibrations and trigonometric formulas of spherical geometry, are found eigenvalues of Laplacian for smooth real-valued functions and corresponding spherical functions for all simply connected CROSS'es (compact rank one Riemannian symmetric spaces). Also in \cite{Ber2} are found direct connections of these functions to special functions such as hypergeometric finite Gauss series, Jacobi polynomials and orthogonal polynomials, including ultraspherical Gegenbauers polynomials, which in turn include Legendre polynomials and Chebyshev polynomials of the first and the second kind.

In our paper by means of the search algorithms for Laplacian spectrum from \cite{BerSvir2} and
\cite{Svir3} we conduct explicit calculations of Laplacian spectrum for smooth real- or
complex-valued functions on all compact connected simple rank three Lie groups with biinvariant Riemannian metric and set a connection of obtained formulas to the number theory and integer ternary and binary quadratic forms.

\section{Preliminaries}

Let $G$ be a compact connected simple Lie group with biinvariant Riemannian metric $\nu$. The set  $\Spec(G,\nu)$ of all eigenvalues of Laplace-Beltrami operator $\Delta$ on smooth real-valued functions defined on $(G,\nu)$ with taking into account of multiplicity of eigenvalues, i.e. dimension of spaces of corresponding eigenfunctions, is called {\it the spectrum} of Laplace operator. In work \cite{BerSvir1} are presented some general notions and results on Laplace-Beltrami operator, i.e. its eigenvalues and eigenfunctions on Riemannian $C^{\infty}$-manifolds. In particular, the spectrum of Lie group $(G,\nu)$ can be presented as follows:
\begin{equation}\label{Eq:spec}
\Spec(G,\nu)=\{0=\lambda_0> \lambda_1\geq \lambda_2\geq\ldots\}.
\end{equation}
The Laplacian is naturally generalized onto complex-valued functions.

\begin{definition}
Bilinear (symmetric) form $k_{\rho}$ on a Lie algebra $\frak{g}$ defined by formula
$$k_{\rho}(u,v)=\trace(\rho(u)\rho(v)),\quad u,\,v\in\frak{g},$$
is said to be the form associated with a representation $\rho$ of $\frak{g}$. The form $k_{\ad}$, where $\ad(u)(v):=[u,v]$ is the adjoint representation of Lie algebra $\frak{g}$, is called the Killing form of Lie algebra $\frak{g}$.
\end{definition}

\begin{remark}
A compact connected Lie group $G$ is simple if and only if its Lie algebra $\frak{g}$ is simple, which is equivalent to irreducibility of the adjoint representation $\ad$. In addition for any irreducible non-zero representation $\rho$ of the Lie algebra $\mathfrak{g}$ the form $k_{\rho}$ is negatively defined and is proportional to the scalar product $\nu$.
\end{remark}

Theorem 1.2 and corollary 1.2 of paper \cite{BerSvir2} implies the following statement.
\begin{proposition}\label{Prop:spec_ad}
If $(G,\nu)$ is a compact connected simple $m$-dimensional Lie group with biinvariant Riemannian metric $\nu$ such that $\nu(e)=-k_{ad}$ then for the adjoint representation $\Ad$ the Lie group
$G$:
\begin{equation}\label{Eq:spec_ad}
\lambda_{\Ad}=-1,\, \dim \Ad = d_{\Ad} = m.
\end{equation}
\end{proposition}

Omitting details, let us remind that a (simple) Lie algebra $\frak{g}$ defines {\it the root system} $\Gamma$ as some subset of the dual space $\frak{t}(\mathbb{R})^{\ast}$ to real form
$\frak{t}(\mathbb{R})$ of the Cartan subalgebra $\frak{t}$ of complex span $\frak{k}$ of the Lie algebra $\frak{g}$. Subsystems $\Gamma^{+}\subset\Gamma$ of {\it positive roots} and
$\Pi=\{\alpha_1,\dots,\alpha_l\}\subset\Gamma^{+}$ of positive (linearly independent)
{\it simple roots} are defined by some natural way. The pair
$(\frak{t}(\mathbb{R}),\left(\cdot,\cdot\right))$ is a real Euclidean space, where
$\left(\cdot,\cdot\right):=k_{\ad}\mid_{\mathfrak{t}(\mathbb{R})}$ is the restriction of the Killing form $k_{\ad}$ of the Lie algebra $\mathfrak{g}$ to the subspace
$\mathfrak{t}(\mathbb{R})$. Preserving the same notation the form $\left(\cdot,\cdot\right)$ can be transfer to the dual space $\frak{t}(\mathbb{R})^{\ast}$. More detailed information on the root and weight systems of Lie algebras is contained in $\S\,5$ \cite{BerSvir1} or in $\S\,1$ \cite{BerSvir2} as well as in book \cite{Burb}.

Let us define function $\{\cdot,\cdot\}$ in the following way: for
$\alpha\neq 0$, $\beta\in\frak{t}(\mathbb{R})^{\ast}$
$$\{\beta,\alpha\}:=\frac{2(\alpha,\beta)}{(\alpha,\alpha)}.$$
It is known that if $\alpha$, $\beta\in\Gamma$ then $\{\beta,\alpha\}\in\mathbb{Z}$.
{\it The fundamental weights} $\bar{\omega}_1,\dots,\bar{\omega}_l$ of a Lie algebra
$\frak{g}$ are uniquely defined  by the following relations:
\begin{equation}\label{Eq:dual_basis}
\{\bar{\omega}_i,\alpha_j\}=\delta_{ij},\,\,\mbox{where }\delta_{ij} -
\mbox{the Kronecker symbol},\,\,1\leq i,j\leq l.
\end{equation}

Let us present the following result of E.~Cartan (see \cite{Cartan}).

\begin{proposition}\label{Cartan}
The sets of weights $\Lambda(\frak{g})$ and highest weights $\Lambda^{+}(\frak{g})$ of all irreducible complex representations of complex span $\frak{k}$ of a simple Lie algebra $\frak{g}$
with the fundamental weights $\bar{\omega}_1,\dots,\bar{\omega}_l$ are given by the following way:
$$\Lambda(\frak{g})=\{\Lambda\in\frak{t}(\mathbb{R})^{\ast}\,\mid\,\Lambda=\sum\limits_{i=1}^{l}\Lambda_{i}\bar{\omega}_i,\,\,\mbox{where }\Lambda_{i}\in\mathbb{Z}\},$$
$$\Lambda^{+}(\frak{g})=\{\Lambda\in\frak{t}(\mathbb{R})^{\ast}\,\mid\,\Lambda=\sum\limits_{i=1}^{l}\Lambda_{i}\bar{\omega}_i,\,\,\mbox{where }\Lambda_{i}\in\mathbb{Z}\,\,\mbox{and }\Lambda_i\geq 0\}.$$
In addition, up to equivalence, an irreducible complex representation of Lie algebra $\frak{g}$ is defined by its highest weight $\Lambda\in\Lambda^{+}(\frak{g})$.
\end{proposition}

The search algorithm of Laplacian spectrum on a Lie group $(G,\nu)$ is contained in consideration of all irreducible complex representations and calculation of eigenvalues of the Laplacian and their multiplicities via highest weights of these representations. For a simply connected Lie group  $(G,\nu)$ the set of highest weights coincides with the set $\Lambda^{+}(\frak{g})$ of its Lie algebra $\frak{g}$ from proposition \ref{Cartan}. In general case of connected simple Lie group $G$ with biinvariant Riemannian metric $\nu$ the search of Laplacian spectrum reduces to the search of its highest weights $\Lambda^{+}(G)$ of irreducible complex representations.

The root system $\Gamma$ given on the space $\frak{t}(\mathbb{R})^{\ast}$ of Lie algebra
$\frak{g}$ with simple roots $\alpha_1,\ldots,\alpha_l$ defines two closed subgroups
$\Lambda_0(\Gamma)$ and $\Lambda_1(\Gamma)$ of the group $\frak{t}(\mathbb{R})^{\ast}$ with respect to addition: $\Lambda_0(\Gamma)$ is lattice generated by the system $\Gamma$, i.e.
$$\Lambda_0(\Gamma) = \Big\{\, \beta\in \frak{t}(\mathbb{R})^{\ast} \;\bigm|\;\beta= \sum_{j=1}^{l} n_j\alpha_j,\;n_j\in \mathbb{Z},\,\,j= 1,\ldots,l\Big\},$$
$\Lambda_1(\Gamma)$ is subgroup which is dual to the system $\Gamma$ relative to function
$\{\cdot, \cdot\}$:
$$\Lambda_1(\Gamma) = \Big\{\, \beta\in \frak{t}(\mathbb{R})^{\ast} \;\bigm|\;\, \{\beta, \alpha_j \}\in \mathbb{Z},\;j = 1,\ldots,l \Big\}.$$
It follows from relations (\ref{Eq:dual_basis}) that $\Lambda_1(\Gamma)$ coincides with the lattice  $\Lambda(\frak{g})$ from proposition \ref{Cartan}, i.e. it is generated by the fundamental weights:
$$\Lambda_1(\Gamma) = \Lambda(\frak{g}) = \Big\{\, \beta\in \frak{t}(\mathbb{R})^{\ast} \;\bigm|\;\, \beta= \sum_{j=1}^{l} n_j\bar{\omega}_j,\;n_j\in \mathbb{Z},\;j=1,\ldots,l\Big\}.$$

The set of all weights $\Lambda(G)$ of Lie group $(G,\nu)$ is called
{\it characteristic lattice} of Lie group $G$. Let us present the most important properties of characteristic lattice from work \cite{Dyn} and appendix by A.L.~Onishchik in book \cite{Adams}.

\begin{proposition}\label{Prop:Lambda1}
Let $G$ be a compact connected Lie group with the root system $\Gamma$. Then characteristic lattice $\Lambda(G)$ is connected with lattices $\Lambda_0(\Gamma)$ and $\Lambda_1(\Gamma)$ by the following relations:
\begin{equation}\label{Eq:pi1zg}
\Lambda_0(\Gamma)\subseteq\Lambda(G)\subseteq\Lambda_1(\Gamma),
\end{equation}
in addition, in this sequence of additive subgroups of the group $\mathfrak{t}(\mathbb{R})^*$ every preceding group is a subgroup of the next one, and
$$
\Lambda_1(\Gamma)/\Lambda(G) \cong \pi_1(G), \quad\quad
\Lambda(G)/\Lambda_0(\Gamma) \cong C(G),
$$
where $\pi_1(G)$ is the fundamental group of $G$, while $C(G)$ is the center of the Lie group $G$.
Moreover, for any additive subgroup $\Lambda$ of the group $\mathfrak{t}(\mathbb{R})^*,$ satisfying the relation (\ref{Eq:pi1zg}), there exists a unique up to isomorphism Lie group $G$, whose characteristic lattice $\Lambda(G)$ coincides with $\Lambda$.
\end{proposition}

Let us consider the family of compact connected simple Lie groups with Lie algebra $\mathfrak{g}$, which has a root system $\Gamma$. In consequence of proposition \ref{Prop:Lambda1} this family,
up to isomorphisms of Lie groups, coincides with the family of all possible lattices of maximal rank, satisfying the relation (\ref{Eq:pi1zg}). Thus, maximal fundamental group
$\Lambda_1(\Gamma)/\Lambda_0(\Gamma)$ plays the main role in classification of simple Lie groups with given (simple) Lie algebra. Let us present all maximal fundamental groups for simple (non-abelian) Lie algebras from appendix by A.L.Onishchik in book \cite{Adams}:
\begin{table}[h]
\begin{center}
\begin{tabular}{|c|c|c|c|c|c|c|}
\hline
Type  $\Gamma$&$A_l$&$B_l, C_l, E_7\;$&$D_{2s}$&$D_{2s+1}$&$E_6$&$E_8, F_4, G_2$\\
\hline
$\Lambda_1(\Gamma)/\Lambda_0(\Gamma)$&$\;\mathbb{Z}_{l+1}$&$\mathbb{Z}_2$&$\mathbb{Z}_2\oplus\mathbb{Z}_2$&$\mathbb{Z}_4$&$\;\mathbb{Z}_3$&$0$\\
\hline
\end{tabular}
\Table\label{Tab:max_gr} Maximal fundamental groups
\end{center}
\end{table}

We get from proposition \ref{Prop:Lambda1} the following

\begin{corollary}\label{Col:LCalc}
Let maximal fundamental group $\Lambda_1(\Gamma)/\Lambda_0(\Gamma)$ of Lie algebra $\mathfrak{g}$ has prime order. Then the family of non-isomorphic compact connected Lie groups with Lie algebra
$\mathfrak{g}$ consists of two groups: simply connected Lie group $G_1$ with the center
$\Lambda_1(\Gamma)/\Lambda_0(\Gamma)$ and weight lattice $\Lambda_1(\Gamma)$, coinciding with weight lattice of Lie algebra $\Lambda(\mathfrak{g})$, and Lie group $G_0$ without center, with fundamental group $\Lambda_1(\Gamma)/\Lambda_0(\Gamma)$ and weight lattice $\Lambda_0(\Gamma)$.
\end{corollary}

We see from description of all irreducible root systems in Tables~I--IX from \cite{Burb} that there are only three irreducible systems of rank three: $A_3 \cong D_3$, $B_3$ и $C_3$. We get from table \ref{Tab:max_gr} above, proposition \ref{Prop:Lambda1}, as well as from the argument after theorem  3 in appendix by A.L.~Onishchuk in book \cite{Adams} and main characterizations of classical matrix Lie groups the following proposition.

\begin{proposition}\label{Prop:list}
The list of Lie groups, corresponding to root systems $A_3$, $B_3$, and $C_3$ looks as follows:
\end{proposition}
\begin{table}[h]
\centering
\begin{tabular}{|p{6cm}|c|c|c|c|c|}
\hline
\centering{Lie group $G$}&$\Gamma$&$\Lambda(G)$&$\pi_1(G)$&$C(G)$&$\dim G$\\
\hline
\centering{$\SU(4)$}&$A_3$&$\Lambda_1$&$0$&$\mathbb{Z}_4$&$15$\\
\hline
\centering{$\SU(4)/(\pm E_4)$} &$A_3$&$\Lambda$, where  $\Lambda_0\subsetneq\Lambda\subsetneq\Lambda_1$&$\mathbb{Z}_2$&$\mathbb{Z}_2$&$15$\\
\hline
\centering{$\PSU(4):=\SU(4)/C(\SU(4))$}&$A_3$&$\Lambda_0$&$\mathbb{Z}_4$&$0$&$15$\\
\hline
\centering{$\Spin(7)$}&$B_3$&$\Lambda_1$&$0$&$\mathbb{Z}_2$&$21$\\
\hline
\centering{$\SO(7)$}&$B_3$&$\Lambda_0$&$\mathbb{Z}_2$&$0$&$21$\\
\hline	
\centering{$\Sp(3)$}&$C_3$&$\Lambda_1$&$0$&$\mathbb{Z}_2$&$21$\\
\hline
\centering{$\PSp(3):=\Sp(3)/C(\Sp(3))$}&$C_3$&$\Lambda_0$&$\mathbb{Z}_2$&$0$&$21$\\
\hline
\end{tabular}
\Table\label{Tab:list} Compact connected simple rank three Lie groups
\end{table}

The usage of above stated properties of characteristic lattice permits to present the calculation algorithm of Laplacian spectra of all Lie groups $(G, \nu)$ with fixed Lie algebra, stated in \cite{Svir3} (see corollary 5), using tables~I--IX from \cite{Burb}, where $\rho$ denotes the vector $\beta$.

\begin{theorem}
\label{alg}
To calculate Laplacian spectra for all compact connected simple Lie groups $G$ with simple Lie algebra $\frak{g}$ with root system $\Gamma$ and biinvariant Riemannian metric $\nu$ satisfying condition $\nu(e)=-\gamma k_{ad}$, where $\gamma>0$, one needs to fulfil the following actions:

1) calculate expression
$b=\langle\tilde{\alpha}+\beta,\tilde{\alpha}+\beta\rangle-\langle\beta,\beta\rangle$,
where $\tilde{\alpha}$ is highest (maximal) root,
assuming that relative to scalar product $\langle\cdot,\cdot\rangle$ (on $\frak{t}(\mathbb{R})$) vectors $\epsilon_i$ from corresponding table in \cite{Burb} are mutually orthogonal and unitary;

2) take scalar product $\left(\cdot,\cdot\right)=\frac{1}{b}\langle\cdot,\cdot\rangle$;

3) find the fundamental weights $\bar{\omega}_1,\dots,\bar{\omega}_l$ of the Lie algebra $\frak{g}$ of the Lie group $G$ (if $\frak{g}$ has rank $l$) by corresponding table from \cite{Burb};

4) for any highest weight $\Lambda\in\Lambda^{+}_1(\Gamma)$, i.e for any
$\Lambda=\sum\limits_{i=1}^{l}\Lambda_{i}\bar{\omega}_i$, where $\Lambda_{i}~\in~\mathbb{Z}$ and
$\Lambda_{i}\geq 0$ for $i=1,\dots,l$, find eigenvalue $\lambda(\Lambda)$ of the Laplace operator, corresponding to highest weight $\Lambda$, by formula
\begin{equation}
\label{lam}
\lambda(\Lambda)=-\frac{1}{\gamma}\Big[\langle\Lambda+\beta,\Lambda+\beta\rangle-\langle\beta,\beta\rangle\Big]
\end{equation}
and dimension $d(\Lambda+\beta)$ of irreducible complex representation of complex span of the Lie algebra $\frak{g}$ with highest weight $\Lambda$ by formula
\begin{equation}
\label{dim}
d(\Lambda+\beta)=\prod\limits_{\alpha\in\Gamma^{+}}\frac{\left(\Lambda+\beta,\alpha\right)}{{\left(\beta,\alpha\right)}};
\end{equation}

5) for any lattice $\Lambda$, satisfying relation
$\Lambda_0(\Gamma)\subseteq\Lambda\subseteq\Lambda_1(\Gamma)$, where $\Lambda_0(\Gamma)$ and
$\Lambda_1(\Gamma)$ are lattices generated by simple roots and fundamental weights, obtained in p.~1) and p.~3) respectively, $G$ is Lie group with Lie algebra $\frak{g}$, corresponding to characteristic lattice $\Lambda=\Lambda(G)$, fulfil the following three actions:

6) find the set of highest weights $\Lambda^{+}(G)=\Lambda(G)\cap\Lambda^{+}_1(\Gamma)$, defining it via fundamental weights $\bar{\omega}_1,\dots,\bar{\omega}_l$;

7) for any highest weight $\Lambda\in\Lambda^{+}(G)$ find from p.~4) eigenvalue
$\lambda(\Lambda)$ and dimension $d(\Lambda+\beta)$ of irreducible complex representation corresponding to the weight $\Lambda$;

8) find multiplicity of any eigenvalue $\lambda=\lambda(\Lambda)$ by formula
\begin{equation}
\label{abc}
\sigma(\lambda)=\sum\limits_{\Lambda:\,\lambda(\Lambda)=\gamma\lambda}\,\prod\limits_{\;\alpha\in\Gamma^{+}}
\left(\frac{\left(\Lambda+\beta,\alpha\right)}{\left(\beta,\alpha\right)}\right)^2,
\end{equation}
obtaining in this way the spectrum $\Spec(G,\nu)$ of the Lie group $G$, corresponding to characteristic lattice $\Lambda(G)$.

Thus, we get all spectra $\Spec(G,\nu)$ of Lie groups $G$ with Lie algebra $\frak{g}$ and metric
$\nu$.
\end{theorem}

\begin{remark}
In formulas (\ref{dim}) and (\ref{abc}) applied in p.~4) and p.~8) of theorem \ref{alg} we can use instead of $\left(\cdot,\cdot\right)$ every scalar product proportional to it, in particular,
$\langle\cdot,\cdot\rangle$ from p.~1 of theorem \ref{alg}.
\end{remark}

Using theorem \ref{alg} we shall find in next sections Laplacian spectra of all compact connected simple rank three Lie groups given in table \ref{Tab:list} above.

\section{Calculation of Laplacian spectrum for Lie groups \newline $\SU(4)$, $\SU(4)/(\pm E_4),$ and $\SU(4)/C(\SU(4))$}\label{Sec:A3}

\begin{remark}
Notice that $C(\SU(4))\subset \SU(4)$ is cyclic subgroup of $4$-th order generated by matrix $iE_4$, where $E_4$ is unit $(4\times 4)$-matrix. Isomorphism $A_3\cong D_3$ implies isomorphisms of Lie groups $\SU(4)\cong \Spin(6)$, $\SU(4)/(\pm E_4)\cong \SO(6)$,
$\SU(4)/C(SU(4))\cong \SO(6)/(\pm E_6).$
\end{remark}

By table \ref{Tab:list}, to Lie groups under consideration corresponds the root system $A_3$. We apply table~I from \cite{Burb}. Simple roots are
$\alpha_1=\varepsilon_1-\varepsilon_2,\,\alpha_2=\varepsilon_2-\varepsilon_3,\,\alpha_3=\varepsilon_3-\varepsilon_4;$
maximal root is $\tilde{\alpha}=\alpha_1+\alpha_2+\alpha_3=\varepsilon_1-\varepsilon_4$.
Positive roots are $\varepsilon_1-\varepsilon_2$, $\varepsilon_2-\varepsilon_3$, $\varepsilon_3-\varepsilon_4$, $\varepsilon_1-\varepsilon_3$,
$\varepsilon_2-\varepsilon_4$, $\varepsilon_1-\varepsilon_4$.
The sum of positive roots is equal to
$2\beta=3\varepsilon_1+\varepsilon_2-\varepsilon_3-3\varepsilon_4$, whence
\begin{equation}
\label{beta1}
\beta=\frac{3\varepsilon_1+\varepsilon_2-\varepsilon_3-3\varepsilon_4}{2},\quad \tilde{\alpha}+\beta=\frac{5\varepsilon_1+\varepsilon_2-\varepsilon_3-5\varepsilon_4}{2}.
\end{equation}

We act according to algorithm suggested in theorem \ref{alg}.

1) $b=\langle\tilde{\alpha}+\beta,\tilde{\alpha}+\beta\rangle-\langle\beta,\beta\rangle=13-5=8.$

2) $\left(\cdot,\cdot\right)=\frac{1}{8}\langle\cdot,\cdot\rangle.$

3) Fundamental weights have the following form
$$\bar{\omega}_1=\frac{3\varepsilon_1-\varepsilon_2-\varepsilon_3-\varepsilon_4}{4},\,\,
\bar{\omega}_2=\frac{\varepsilon_1+\varepsilon_2-\varepsilon_3-\varepsilon_4}{2},\,\,
\bar{\omega}_3=\frac{\varepsilon_1+\varepsilon_2+\varepsilon_3-3\varepsilon_4}{4}.$$

It is easy to see that
$$\alpha_1=2\bar{\omega}_1-\bar{\omega}_2,\quad\alpha_2=2\bar{\omega}_2-\bar{\omega}_1-\bar{\omega}_3,\quad
\alpha_3=2\bar{\omega}_3-\bar{\omega}_2,\quad\tilde{\alpha}=\bar{\omega}_1+\bar{\omega}_3,\quad\beta=\bar{\omega}_1+\bar{\omega}_2+\bar{\omega}_3.$$

4) Let
$\Lambda=\sum\limits_{i=1}^{3}\Lambda_i\bar{\omega}_i,$ where $\Lambda_i\in\mathbb{Z}$,
$\Lambda_i\geq 0$, $i=1,2,3$. Then
$$\Lambda+\beta=\sum\limits_{i=1}^{3}(\Lambda_i+1)\bar{\omega}_i=
\sum\limits_{i=1}^{3}\nu_i\bar{\omega}_i=$$
\begin{equation}
\label{lam1}
=\frac{1}{4}\left[(3\nu_1+2\nu_2+\nu_3)\varepsilon_1+(-\nu_1+2\nu_2+\nu_3)\varepsilon_2+
(-\nu_1-2\nu_2+\nu_3)\varepsilon_3-(\nu_1+2\nu_2+3\nu_3)\varepsilon_4\right],
\end{equation}
where
$$\nu_i=\Lambda_i+1,\quad\nu_i\in\mathbb{N},\quad i=1,2,3.$$

By formula (\ref{lam}), eigenvalue $\lambda(\Lambda),$ corresponding to highest weight
$\Lambda,$ is equal to
$$\lambda(\Lambda)=-\frac{1}{8\gamma}[\langle\Lambda+\beta,\Lambda+\beta\rangle-\langle\beta,\beta\rangle]=$$
$$=-\frac{1}{128\gamma}\left((3\nu_1+2\nu_2+\nu_3)^2+(-\nu_1+2\nu_2+\nu_3)^2+(-\nu_1-2\nu_2+\nu_3)^2+(\nu_1+2\nu_2+3\nu_3)^2-80\right).$$
After routine calculations we get that
\begin{equation}
\label{a1}
\lambda(\Lambda)=-\frac{1}{32\gamma} \left((\nu_1+2\nu_2+\nu_3)^2+2\nu_1^2+2\nu_3^2-20\right).
\end{equation}

Calculation by formula (\ref{dim}) of dimension $d(\Lambda+\beta)$ of representation
$\rho(\Lambda)$ corresponding to highest weight $\Lambda$ gives
$$d(\Lambda+\beta)=\frac{(\Lambda+\beta,\varepsilon_1-\varepsilon_2)}{(\beta,\varepsilon_1-\varepsilon_2)}\cdot
\frac{(\Lambda+\beta,\varepsilon_2-\varepsilon_3)}{(\beta,\varepsilon_2-\varepsilon_3)}\cdot
\frac{(\Lambda+\beta,\varepsilon_3-\varepsilon_4)}{(\beta,\varepsilon_3-\varepsilon_4)}\times$$
$$\times\frac{(\Lambda+\beta,\varepsilon_1-\varepsilon_3)}{(\beta,\varepsilon_1-\varepsilon_3)}
\cdot\frac{(\Lambda+\beta,\varepsilon_2-\varepsilon_4)}{(\beta,\varepsilon_2-\varepsilon_4)}\cdot\frac{(\Lambda+\beta,\varepsilon_1-\varepsilon_4)}{(\beta,\varepsilon_1-\varepsilon_4)}.$$
On the ground of (\ref{beta1}) and (\ref{lam1}),
$$\frac{(\Lambda+\beta,\varepsilon_1-\varepsilon_2)}{(\beta,\varepsilon_1-\varepsilon_2)}=\nu_1,\quad
\frac{(\Lambda+\beta,\varepsilon_2-\varepsilon_3)}{(\beta,\varepsilon_2-\varepsilon_3)}=\nu_2,\quad
\frac{(\Lambda+\beta,\varepsilon_3-\varepsilon_4)}{(\beta,\varepsilon_3-\varepsilon_4)}=\nu_3,$$
$$\frac{(\Lambda+\beta,\varepsilon_1-\varepsilon_3)}{(\beta,\varepsilon_1-\varepsilon_3)}=\frac{\nu_1+\nu_2}{2},\quad
\frac{(\Lambda+\beta,\varepsilon_2-\varepsilon_4)}{(\beta,\varepsilon_2-\varepsilon_4)}=\frac{\nu_2+\nu_3}{2},\quad
\frac{(\Lambda+\beta,\varepsilon_1-\varepsilon_4)}{(\beta,\varepsilon_1-\varepsilon_4)}=\frac{\nu_1+\nu_2+\nu_3}{3}.$$
Therefore
\begin{equation}
\label{a2}
d(\Lambda+\beta)=\frac{1}{12}\;\nu_1\nu_2\nu_3(\nu_1+\nu_2)(\nu_2+\nu_3)(\nu_1+\nu_2+\nu_3).
\end{equation}

\begin{remark}
It follows from (\ref{a1}) and (\ref{a2}) that $\lambda(\Lambda)=0 \Leftrightarrow \Lambda = 0\Leftrightarrow\nu_i=1$, $i=1,2,3$,
and hence, the multiplicity $\sigma(\lambda(0))=d^2(0+\beta)=1$, what corresponds to formula  (\ref{Eq:spec}).
In addition, on the ground of proposition \ref{Prop:spec_ad}, the highest weight of the adjoined representation $\Ad$ is the maximal root
$\tilde{\alpha}=\alpha_1+\alpha_2+\alpha_3=\bar{\omega}_1+\bar{\omega}_3$, for which
$\nu_1 = \nu_3 = 2$ и $\nu_2 = 1$. Therefore eigenvalue and dimension of representation $\Ad$ can be calculated as follows:
$$\lambda_{\Ad}=\lambda(\tilde{\alpha})= -\frac{1}{32\gamma}\left(36+8+8-20\right) = -\frac{1}{\gamma},$$
$$\dim\Ad=d(\tilde{\alpha}+\beta)=\frac{2\cdot 2\cdot3\cdot3\cdot5}{12}=15,$$
what corresponds to formula (\ref{Eq:spec_ad}) for $\gamma=1$ with taking into account the equality $\dim \SU(4)=15$ from Table \ref{Tab:list}.
\end{remark}

5) Simple roots and fundamental weights indicated above generate respective lattices
$\Lambda_0(A_3)$ and $\Lambda_1(A_3)$, i.e.
\begin{equation}
\label{a3}
\Lambda_0(A_3)=\left\{\sum\limits_{i=1}^{3}\Psi_i\alpha_i\,\Bigm|\,\Psi_1,\Psi_2,\Psi_3\in\mathbb{Z}\right\},\;\;
\Lambda_1(A_3)=\left\{\sum\limits_{i=1}^{3}\Lambda_i\bar{\omega}_i\,\Bigm|\,\Lambda_1,\Lambda_2,\Lambda_3\in\mathbb{Z}\right\}.
\end{equation}

After expressing roots via fundamental weights
$$\alpha_1=2\bar{\omega}_1-\bar{\omega}_2,\quad \alpha_2=-\bar{\omega}_1+2\bar{\omega}_2-\bar{\omega}_3,\quad \alpha_3=-\bar{\omega}_2+2\bar{\omega}_3$$
and the change of variables
$$\left\{
\begin{array}{l}
\Psi_1=\Omega_1+\Omega_3, \\
\Psi_2=\Omega_1+2\Omega_3, \\
\Psi_3=\Omega_1-\Omega_2+3\Omega_3
\end{array}\right.\quad\Leftrightarrow\quad
\left\{
\begin{array}{l}
\Omega_1=2\Psi_1-\Psi_2, \\
\Omega_2=-\Psi_1+2\Psi_2-\Psi_3, \\
\Omega_3=-\Psi_1+\Psi_2
\end{array}\right.$$
the lattice $\Lambda_0(A_3)$ takes the following view
\begin{equation}
\label{a4}
\Lambda_0(A_3)=\{\Omega_1(\bar{\omega}_1+\bar{\omega}_3)+\Omega_2(\bar{\omega}_2-2\bar{\omega}_3)+4\Omega_3\bar{\omega}_3\,\,\mid\,\Omega_1,\Omega_2,\Omega_3\in\mathbb{Z}\}.
\end{equation}
Also under the change of variables
$$ \Lambda_3^{'} = \Lambda_3 + 2\Lambda_2 - \Lambda_1 \quad\Leftrightarrow\quad \Lambda_3 = \Lambda_1 - 2\Lambda_2 + \Lambda_3^{'} $$
the lattice $\Lambda_1(A_3)$ takes the following form
\begin{equation}
\label{a5}
\Lambda_1(A_3)=\{\Lambda_1(\bar{\omega}_1+\bar{\omega}_3)+\Lambda_2(\bar{\omega}_2-2\bar{\omega}_3)+\Lambda_3^{'}\bar{\omega}_3\,\,\mid\,\Lambda_1,\Lambda_2,\Lambda_3^{'}\in\mathbb{Z}\}.
\end{equation}

It follows from (\ref{a4}) and (\ref{a5}) that $\Lambda_0(A_3)\subset\Lambda_1(A_3)$ and
$\Lambda_1(A_3)/\Lambda_0(A_3)\cong\mathbb{Z}_4$ and there exists a lattice $\Lambda(A_3)$, for which $\Lambda_0(A_3)\subsetneq\Lambda(A_3)\subsetneq\Lambda_1(A_3)$, of the following form
\begin{equation}
\label{a6}
\Lambda(A_3)=\{\Omega_1(\bar{\omega}_1+\bar{\omega}_3)+\Omega_2(\bar{\omega}_2-2\bar{\omega}_3)+2\Omega_3\bar{\omega}_3\,\,\mid\,\Omega_1,\Omega_2,\Omega_3\in\mathbb{Z}\}.
\end{equation}
Lie groups, corresponding to these characteristic lattices, are given in table \ref{Tab:list}.
The sets of highest weights, corresponding to these lattices, are depicted on figure \ref{fig:lat_a} above.
\begin{figure}
\includegraphics[width=0.9\textwidth]{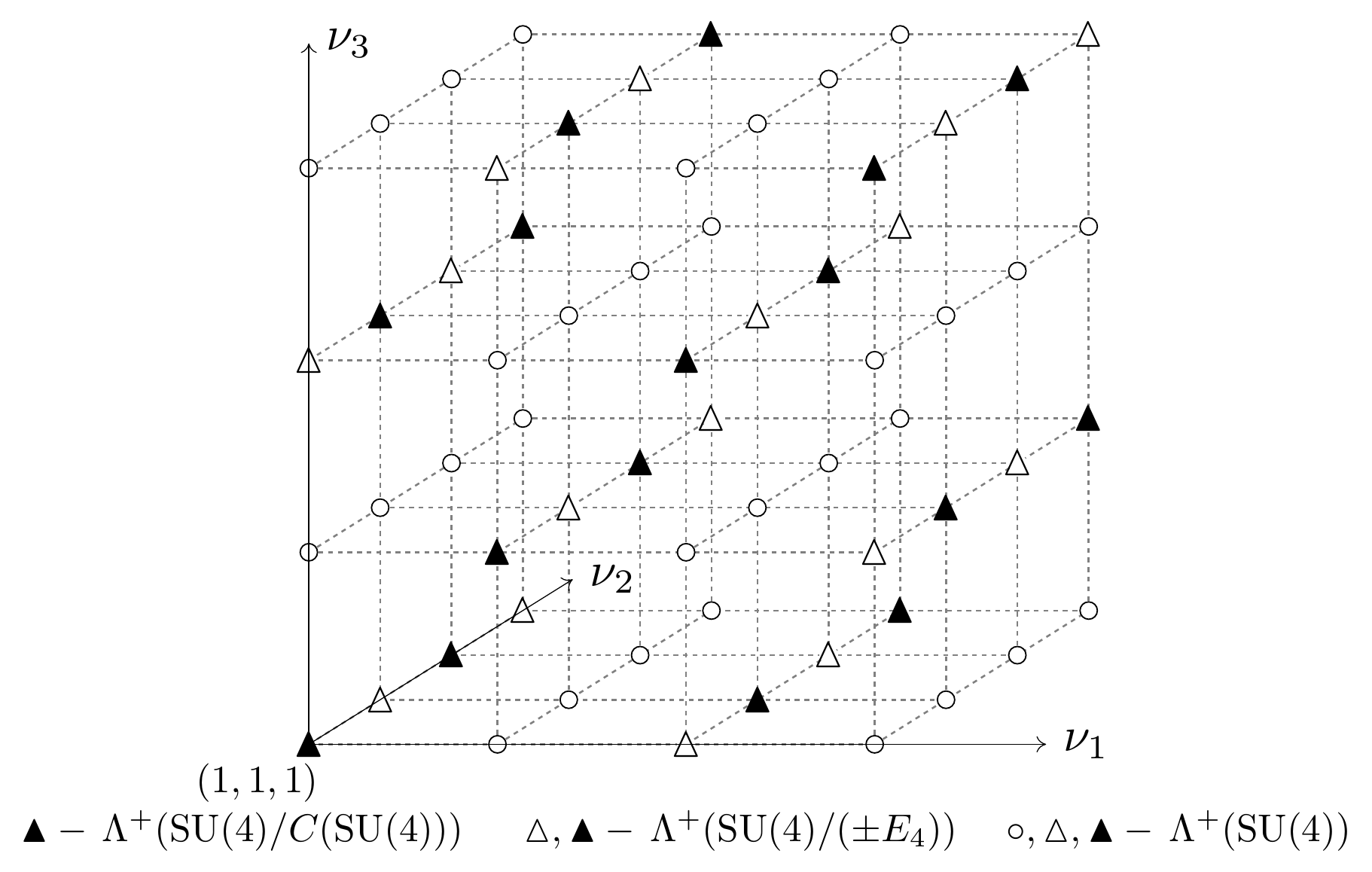}
\caption{The sets of highest weights of Lie groups, corresponding to the root system $A_3$}
\label{fig:lat_a}
\end{figure}

a) Let us present formulas for Laplacian spectrum of the Lie group $\SU(4)$.

6a) By formula (\ref{a3}) we get the following set of highest weights of Lie group $\SU(4)$
$$\Lambda^{+}(\SU(4))=\Lambda_1^{+}(A_3)=\left\{\sum\limits_{i=1}^{3}\Lambda_i\bar{\omega}_i\,\Bigm|\,\Lambda_i\in\mathbb{Z},\,\,\Lambda_i\geq 0,\,\,i=1,2,3\right\}.$$

7a) Set $\Lambda=\sum\limits_{i=1}^{3}\Lambda_i\bar{\omega}_i\in\Lambda^{+}(\SU(4))$, then by p.~4) eigenvalue $\lambda(\Lambda)$ and dimension $d(\Lambda+\beta)$ are calculated by formulas (\ref{a1}) and (\ref{a2}) respectively, where $\nu_i=\Lambda_i+1$, $\nu_i\in\mathbb{N}$, $i=1,2,3.$

8a) Applying formula (\ref{abc}) and results of preceding point, we get the following multiplicity of eigenvalue $\lambda(\Lambda)$
$$\sigma(\Lambda)=\frac{1}{144}\sum_{\;\Xi_1}
\nu_1^2\nu_2^2\nu_3^2(\nu_1+\nu_2)^2(\nu_2+\nu_3)^2(\nu_1+\nu_2+\nu_3)^2,$$
where
\begin{equation}
\label{xi1}
\Xi_1=\{(\nu_1,\nu_2,\nu_3)\in\mathbb{N}^3\,\mid\,(\nu_1+2\nu_2+\nu_3)^2+2\nu_1^2+2\nu_3^2=20-32\gamma\lambda\}.
\end{equation}

The least by modulus non-zero eigenvalue of Laplacian is equal to $-\frac{15}{32\gamma}$ and corresponds to irreducible complex representations of the Lie group $\SU(4)$ with highest weights  $\bar{\omega}_1$ and $\bar{\omega}_3$. Dimensions of these representations are equal to $4$. Therefore the multiplicity of the eigenvalue $-\frac{15}{32\gamma}$ is equal to $4^2+4^2=32$.

b) Let us give formulas for Laplacian spectrum of the Lie group $\SU(4)/(\pm E_4)$, where $E_4$
is unit matrix of the fourth order.

6b) By formula (\ref{a6}), we get the following set of highest weights of the Lie group
$\SU(4)/(\pm E_4)$
\begin{equation*}
\begin{array}{r c l}
\Lambda^{+}(\SU(4)/Z_2)&=&\Lambda(A_3)\cap\Lambda^{+}_1(A_3)=\{\Omega_1\bar{\omega}_1+\Omega_2\bar{\omega}_2+(\Omega_1-2\Omega_2+2\Omega_3)\bar{\omega}_3\mid\\
&&\Omega_1,\Omega_2,\Omega_3\in\mathbb{Z},\,\Omega_1\geq 0,\,\Omega_2\geq 0,\,\Omega_1-2\Omega_2+2\Omega_3\geq 0\}.
\end{array}
\end{equation*}

7b) Set $\Lambda=\Omega_1\bar{\omega}_1+\Omega_2\bar{\omega}_2+(\Omega_1-2\Omega_2+2\Omega_3)\bar{\omega}_3\in\Lambda^{+}(\SU(4)/(\pm E_4))$, then by p.~4) eigenvalue $\lambda(\Lambda)$ and dimension $d(\Lambda+\beta)$ are calculated by formulas (\ref{a1}) and (\ref{a2}) respectively, where
$\nu_1=\Omega_1+1$, $\nu_2=\Omega_2+1$, $\nu_3=\Omega_1-2\Omega_2+2\Omega_3+1$, $\nu_i\in\mathbb{N}$, $i=1,2,3,$ $\nu_3\equiv \nu_1(\Mod 2)$.

8b) Applying formula (\ref{abc}) and results of preceding point, we get the following multiplicity of the eigenvalue $\lambda(\Lambda)$
$$\sigma(\Lambda)=\frac{1}{144}\sum_{\;\Xi_2}
\nu_1^2\nu_2^2\nu_3^2(\nu_1+\nu_2)^2(\nu_2+\nu_3)^2(\nu_1+\nu_2+\nu_3)^2,$$
where
\begin{equation}
\label{xi2}
\Xi_2=\{(\nu_1,\nu_2,\nu_3)\in\mathbb{N}^3\mid\,(\nu_1+2\nu_2+\nu_3)^2+2\nu_1^2+2\nu_3^2=20-32\gamma\lambda,\,\,\nu_3\equiv \nu_1(\Mod 2)\}.
\end{equation}

The least by modulus non-zero eigenvalue of Laplacian is equal to $-\frac{5}{8\gamma}$ and corresponds to irreducible complex representation of the Lie group $\SU(4)/(\pm E_4)$ with highest weight $\bar{\omega}_2$. The dimension of this representation is equal to $6$. Then the multiplicity of the eigenvalue $-\frac{5}{8\gamma}$ is equal to $6^2=36$.

c) Let us give formulas for Laplacian spectrum of the Lie group $\SU(4)/C(\SU(4))$.

6c) By formula (\ref{a4}), we get the set of highest weights of the Lie group $\SU(4)/C(\SU(4))$\begin{equation*}
\begin{array}{r c l}
\Lambda^{+}(\SU(4)/C(\SU(4)))&=&\Lambda_0(A_3)\cap\Lambda^{+}_1(A_3)=\{\Omega_1\bar{\omega}_1+\Omega_2\bar{\omega}_2+(\Omega_1-2\Omega_2+4\Omega_3)\bar{\omega}_3\mid\\
&&\Omega_1,\Omega_2,\Omega_3\in\mathbb{Z},\,\Omega_1\geq 0,\,\Omega_2\geq 0,\,\Omega_1-2\Omega_2+4\Omega_3\geq 0\}.
\end{array}
\end{equation*}

7c) Set $\Lambda=\Omega_1\bar{\omega}_1+\Omega_2\bar{\omega}_2+(\Omega_1-2\Omega_2+4\Omega_3)\bar{\omega}_3\in\Lambda^{+}(\SU(4)/C(\SU(4)))$, then by p.~4) eigenvalue $\lambda(\Lambda)$ and
dimension $d(\Lambda+\beta)$ are calculated by formulas (\ref{a1}) and (\ref{a2}) respectively, where $\nu_1=\Omega_1+1$, $\nu_2=\Omega_2+1$, $\nu_3=\Omega_1-2\Omega_2+4\Omega_3+1$, $\nu_i\in\mathbb{N}$, $i=1,2,3,$ $\nu_3-\nu_1+2\nu_2\equiv 2(\Mod 4)$.

8c) Applying formula (\ref{abc}) and results of preceding point, we get the following multiplicity of the eigenvalue $\lambda(\Lambda)$
$$\sigma(\Lambda)=\frac{1}{144}\sum_{\;\Xi_3}
\nu_1^2\nu_2^2\nu_3^2(\nu_1+\nu_2)^2(\nu_2+\nu_3)^2(\nu_1+\nu_2+\nu_3)^2,$$
where
\begin{equation}
\label{xi3}
\Xi_3=\{(\nu_1,\nu_2,\nu_3)\in\mathbb{N}^3\mid\,(\nu_1+2\nu_2+\nu_3)^2+2\nu_1^2+2\nu_3^2=20-32\gamma\lambda,\,\,
\nu_3-\nu_1+2\nu_2\equiv 2(\Mod 4)\}.
 \end{equation}

The least by modulus non-zero eigenvalue of Laplacian is equal to $-\frac{1}{\gamma}$ and corresponds to irreducible complex representation of the Lie group $\SU(4)/C(\SU(4))$ with highest weight $\bar{\omega}_1+\bar{\omega}_3$. The dimension of this representation is equal to $15$. Consequently the multiplicity of the eigenvalue $-\frac{1}{\gamma}$ is equal to $15^2=225$.

\section{Calculation of Laplacian spectra for Lie groups $\Spin(7)$ and $\SO(7)$}\label{Sec:B3}

By Table \ref{Tab:list}, to Lie groups under consideration corresponds the root system $B_3$. Let us apply table~II from \cite{Burb}. Simple roots are
$\alpha_1=\varepsilon_1-\varepsilon_2,\,\,\alpha_2=\varepsilon_2-\varepsilon_3,\,\,\alpha_3=\varepsilon_3$; the maximal root is
$\tilde{\alpha}=\alpha_1+2\alpha_2+2\alpha_3=\varepsilon_1+\varepsilon_2$.
Positive roots are: $\varepsilon_i$, $i=1,2,3$; $\varepsilon_1\pm\varepsilon_2$, $\varepsilon_1\pm\varepsilon_3$, $\varepsilon_2\pm\varepsilon_3$.
The sum of positive roots is equal to
$2\beta=5\varepsilon_1+3\varepsilon_2+\varepsilon_3$, whence
\begin{equation}
\label{k2}
\beta=\frac{5\varepsilon_1+3\varepsilon_2+\varepsilon_3}{2},\quad \tilde{\alpha}+\beta=\frac{7\varepsilon_1+5\varepsilon_2+\varepsilon_3}{2}.
\end{equation}

We act according to algorithm given in theorem \ref{alg}.

1) $b=\langle\tilde{\alpha}+\beta,\tilde{\alpha}+\beta\rangle-\langle\beta,\beta\rangle=\frac{75}{4}-\frac{35}{4}=10.$

2) $\left(\cdot,\cdot\right)=\frac{1}{10}\langle\cdot,\cdot\rangle.$

3) Fundamental weights have the form
$$\bar{\omega}_1=\varepsilon_1,\quad
\bar{\omega}_2=\varepsilon_1+\varepsilon_2,\quad
\bar{\omega}_3=\frac{\varepsilon_1+\varepsilon_2+\varepsilon_3}{2}.$$
It is easy to see that
$$\tilde{\alpha}=\bar{\omega}_2,\quad\beta=\bar{\omega}_1+\bar{\omega}_2+\bar{\omega}_3.$$

4) Set
$\Lambda=\sum\limits_{i=1}^{3}\Lambda_i\bar{\omega}_i,$ where $\Lambda_i\in\mathbb{Z}$,
$\Lambda_i\geq 0$, $i=1,2,3$. Then
\begin{equation}
\label{k3}
\Lambda+\beta=\sum_{i=1}^{3}(\Lambda_i+1)\bar{\omega}_i=
\sum_{i=1}^{3}\nu_i\bar{\omega}_i
=\frac{\varepsilon_1}{2}\left(2\nu_1+2\nu_2+\nu_3\right)+
\frac{\varepsilon_2}{2}\left(2\nu_2+\nu_3\right)+\frac{\varepsilon_3}{2}\nu_3,
\end{equation}
where $\nu_i=\Lambda_i+1$, $\nu_i\in\mathbb{N}$, $i=1,2,3.$

By formula (\ref{lam}), eigenvalue $\lambda(\Lambda),$ corresponding to highest weight
$\Lambda,$ is equal to
\begin{equation}
\label{b1}
\lambda(\Lambda)=-\frac{1}{10\gamma}[\langle\Lambda+\beta,\Lambda+\beta\rangle-\langle\beta,\beta\rangle]
=-\frac{1}{40\gamma}\left((2\nu_1+2\nu_2+\nu_3)^2+(2\nu_2+\nu_3)^2+\nu_3^2-35\right).
\end{equation}

By formula (\ref{dim}), dimension $d(\Lambda+\beta)$ of representation
$\rho(\Lambda),$ correspomding to highest weight $\Lambda,$ is equal to
$$d(\Lambda+\beta)=\frac{(\Lambda+\beta,\varepsilon_1)}{(\beta,\varepsilon_1)}\cdot\frac{(\Lambda+\beta,\varepsilon_2)}{(\beta,\varepsilon_2)}\cdot
\frac{(\Lambda+\beta,\varepsilon_3)}{(\beta,\varepsilon_3)}\cdot\frac{(\Lambda+\beta,\varepsilon_1-\varepsilon_2)}{(\beta,\varepsilon_1-\varepsilon_2)}\cdot\frac{(\Lambda+\beta,\varepsilon_1-\varepsilon_3)}{(\beta,\varepsilon_1-\varepsilon_3)}\times$$
$$\times\frac{(\Lambda+\beta,\varepsilon_2-\varepsilon_3)}{(\beta,\varepsilon_2-\varepsilon_3}\cdot
\frac{(\Lambda+\beta,\varepsilon_1+\varepsilon_2)}{(\beta,\varepsilon_1+\varepsilon_2)}
\cdot\frac{(\Lambda+\beta,\varepsilon_1+\varepsilon_3)}{(\beta,\varepsilon_1+\varepsilon_3)}
\cdot\frac{(\Lambda+\beta,\varepsilon_2+\varepsilon_3)}{(\beta,\varepsilon_2+\varepsilon_3)}.$$

On the ground of (\ref{k2}) and (\ref{k3}),
$$\frac{(\Lambda+\beta,\varepsilon_1)}{(\beta,\varepsilon_1)}=
\frac{2\nu_1+2\nu_2+\nu_3}{5},\quad
\frac{(\Lambda+\beta,\varepsilon_2)}{(\beta,\varepsilon_2)}=
\frac{2\nu_2+\nu_3}{3},$$
$$\frac{(\Lambda+\beta,\varepsilon_3)}{(\beta,\varepsilon_3)}=
\nu_3,\quad
\frac{(\Lambda+\beta,\varepsilon_1-\varepsilon_2)}{(\beta,\varepsilon_1-\varepsilon_2)}=
\nu_1,\quad
\frac{(\Lambda+\beta,\varepsilon_1-\varepsilon_3)}{(\beta,\varepsilon_1-\varepsilon_3)}=
\frac{\nu_1+\nu_2}{2},$$
$$\frac{(\Lambda+\beta,\varepsilon_2-\varepsilon_3)}{(\beta,\varepsilon_2-\varepsilon_3)}=
\nu_2,\quad
\frac{(\Lambda+\beta,\varepsilon_1+\varepsilon_2)}{(\beta,\varepsilon_1+\varepsilon_2)}=
\frac{\nu_1+2\nu_2+\nu_3}{4},$$
$$\frac{(\Lambda+\beta,\varepsilon_1+\varepsilon_3)}{(\beta,\varepsilon_1+\varepsilon_3)}=
\frac{\nu_1+\nu_2+\nu_3}{3},\quad
\frac{(\Lambda+\beta,\varepsilon_2+\varepsilon_3)}{(\beta,\varepsilon_2+\varepsilon_3)}=
\frac{\nu_2+\nu_3}{2}.$$
Consequently
\begin{equation}
\label{b2}
d(\Lambda+\beta)=\frac{1}{720}\nu_1\nu_2\nu_3(\nu_1+\nu_2)(\nu_2+\nu_3)(2\nu_2+\nu_3)(\nu_1+\nu_2+\nu_3)(\nu_1+2\nu_2+\nu_3)(2\nu_1+2\nu_2+\nu_3).
\end{equation}

\begin{remark}
It follows from (\ref{b1}) and (\ref{b2}) that
$\lambda(\Lambda)=0 \Leftrightarrow \Lambda = 0\Leftrightarrow\nu_i=1$, $i=1,2,3$ and hence the multiplicity $\sigma(\lambda_0)=d^2(0+\beta)=1$, what corresponds to formula (\ref{Eq:spec}).
Moreover, on the base of proposition \ref{Prop:spec_ad} the highest weight of the adjoined representation $\Ad$ is the maximal root
$\tilde{\alpha}=\alpha_1+2\alpha_2+2\alpha_3=\bar{\omega}_2$,for whom $\nu_1 = \nu_3 = 1$ and
$\nu_2 = 2$. Therefore eigenvalue and dimension of representation $\Ad$ can be calculated as follows:
$$\lambda_{\Ad}=\lambda(\tilde{\alpha})= -\frac{1}{40\gamma}\left(49+25+1-35\right)= -\frac{1}{\gamma},$$
$$\dim\Ad=d(\tilde{\alpha}+\beta)=\frac{2\cdot 3\cdot 3\cdot 5\cdot 4\cdot 6 \cdot 7}{720}=21,$$
that corresponds to formula (\ref{Eq:spec_ad}) for $\gamma=1$ with taking into account the equality  $\dim \Spin(7)=21$ from Table \ref{Tab:list}.
\end{remark}

5) Simple roots and fundamental weights of the Lie algebra $\mathfrak{so}(7)$ indicated above assign respective lattices $\Lambda_0(B_3)$ and $\Lambda_1(B_3)$ of the form
\begin{equation}
\label{b3}
\Lambda_0(B_3)=\left\{\sum\limits_{i=1}^{3}\Psi_i\alpha_i\,\mid\,\Psi_1,\Psi_2,\Psi_3\in\mathbb{Z}\right\},\quad
\Lambda_1(B_3)=\left\{\sum\limits_{i=1}^{3}\Lambda_i\bar{\omega}_i\,\mid\,\Lambda_1,\Lambda_2,\Lambda_3\in\mathbb{Z}\right\}.
\end{equation}

After expressing roots via fundamental weights
$$\alpha_1=2\bar{\omega}_1-\bar{\omega}_2,\quad \alpha_2=-\bar{\omega}_1+2\bar{\omega}_2-2\bar{\omega}_3,\quad \alpha_3=-\bar{\omega}_2+2\bar{\omega}_3$$
and the change of variables
$$\left\{
\begin{array}{l}
\Psi_1=\Omega_1+\Omega_2+\Omega_3, \\
\Psi_2=\Omega_1+2\Omega_2+2\Omega_3, \\
\Psi_3=\Omega_1+2\Omega_2+3\Omega_3
\end{array}\right.\quad\Leftrightarrow\quad
\left\{
\begin{array}{l}
\Omega_1=2\Psi_1-\Psi_2, \\
\Omega_2=-\Psi_1+2\Psi_2-\Psi_3, \\
\Omega_3=-\Psi_2+\Psi_3
\end{array}\right.$$
the lattice $\Lambda_0(B_3)$ takes the following form
\begin{equation}
\label{b4}
\Lambda_0(B_3)=\{\Omega_1\bar{\omega}_1+\Omega_2\bar{\omega}_2+2\Omega_3\bar{\omega}_3\,\,\mid\,\Omega_1,\Omega_2,\Omega_3\in\mathbb{Z}\}.
\end{equation}

It follows from (\ref{b3}) and (\ref{b4}) that $\Lambda_0(B_3)\subset\Lambda_1(B_3)$ and
$\Lambda_1(B_3)/\Lambda_0(B_3)\cong\mathbb{Z}_2$, i.e. it has prime order. Hence on the ground of corollary \ref{Col:LCalc} there is no other lattices. Lie groups, associated with these characteristic lattices, are presented in table \ref{Tab:list}.
The sets of highest weights, corresponding to these lattices, are depicted on figure \ref{fig:lat_b} above.
\begin{figure}
\includegraphics[width=0.55\textwidth]{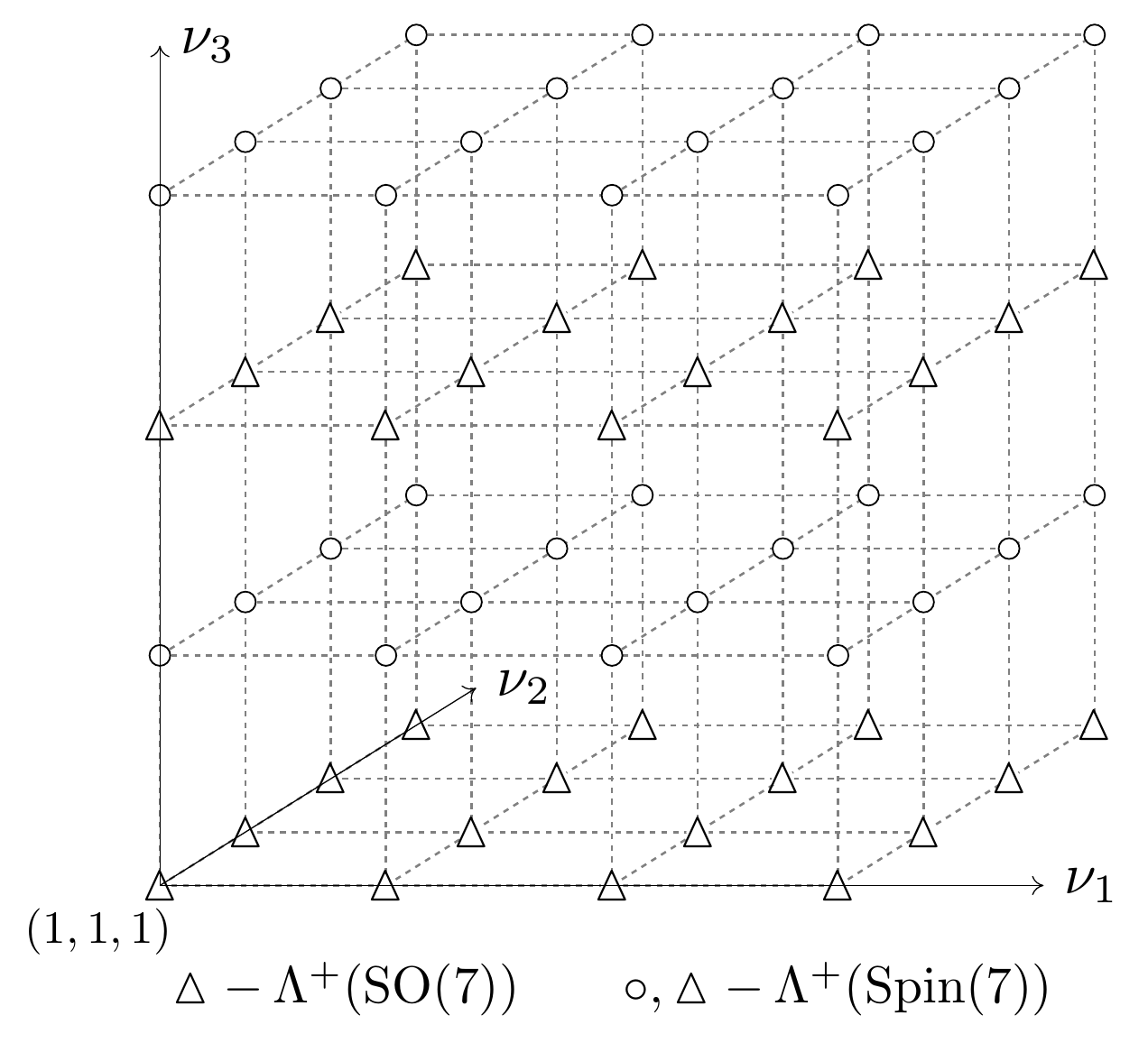}
\caption{The sets of highest weights of Lie groups, corresponding to the root system $B_3$}
\label{fig:lat_b}
\end{figure}

a) Let us give formulas defining Laplacian spectrum of the Lie group $\Spin(7)$.

6a) By formula (\ref{b3}), the set of highest weights of the Lie group $\Spin(7)$ is equal to
$$\Lambda^{+}(\Spin(7))=\left\{\sum\limits_{i=1}^{3}\Lambda_i\bar{\omega}_i\,\mid\,\Lambda_i\in\mathbb{Z},\,\,\Lambda_i\geq 0,\,\,i=1,2,3\right\}.$$

7a) Set $\Lambda=\sum\limits_{i=1}^{3}\Lambda_i\bar{\omega}_i\in\Lambda^{+}(\Spin(7))$, then by p.~4) eigenvalue $\lambda(\Lambda)$ and dimension $d(\Lambda+\beta)$ are calculated by formulas  (\ref{c1}) and (\ref{c2}) respectively, where $\nu_i=\Lambda_i+1$, $\nu_i\in\mathbb{N}$, $i=1,2,3.$

8a) Applying formula (\ref{abc}) and results of preceding point, we get the multiplicity of eigenvalue $\lambda(\Lambda)$:
$$\sigma(\Lambda)=\frac{1}{(720)^2}\sum_{\;\Xi_4}
\nu_1^2\nu_2^2\nu_3^2(\nu_1+\nu_2)^2(\nu_2+\nu_3)^2(2\nu_2+\nu_3)^2(\nu_1+\nu_2+\nu_3)^2(\nu_1+2\nu_2+\nu_3)^2(2\nu_1+2\nu_2+\nu_3)^2,$$
where
\begin{equation}
\label{xi4}
\Xi_4=\{(\nu_1,\nu_2,\nu_3)\in\mathbb{N}^3\,\mid\,(2\nu_1+2\nu_2+\nu_3)^2+(2\nu_2+\nu_3)^2+\nu_3^2=35-40\gamma\lambda\}.
\end{equation}

The least by modulus non-zero eigenvalue of Laplacian is equal to $-\frac{21}{40\gamma}$ and corresponds to irreducible complex representation of the Lie group $\Spin(7)$ with highest weight
$\bar{\omega}_3$. The dimension of this representation is equal to $8$. Therefore the multiplicity of the eigenvalue $-\frac{21}{40\gamma}$ is equal to $8^2=64$.

b) Let us give formulas defining Laplacian spectrum of the Lie group $\SO(7)$.

6b) By formula (\ref{b4}), the set of highest weights of the Lie group $\SO(7)$ is equal to
$$\Lambda^{+}(\SO(7))=\{\Omega_1\bar{\omega}_1+\Omega_2\bar{\omega}_2+2\Omega_3\bar{\omega}_3\,\,\mid\,\Omega_i\in\mathbb{Z},\,\,\Omega_i\geq 0,\,\,i=1,2,3\}.$$

7b) Set $\Lambda=\Omega_1\bar{\omega}_1+\Omega_2\bar{\omega}_2+2\Omega_3\bar{\omega}_3\in\Lambda^{+}(\SO(7))$, then by p.~4) eigenvalue $\lambda(\Lambda)$ and dimension $d(\Lambda+\beta)$ are calculated by formulas (\ref{b1}) and (\ref{b2}) respectively, where
$\nu_1=\Omega_1+1$, $\nu_2=\Omega_2+1$, $\nu_3=2\Omega_3+1$, $\nu_i\in\mathbb{N}$, $i=1,2,3,$ $\nu_3\equiv 1(\Mod 2)$.

8b) Applying formula (\ref{abc}) and results of preceding point, we get the following multiplicity of the eigenvalue $\lambda(\Lambda)$
$$\sigma(\Lambda)=\frac{1}{(720)^2}\sum_{\;\Xi_5}
\nu_1^2\nu_2^2\nu_3^2(\nu_1+\nu_2)^2(\nu_2+\nu_3)^2(2\nu_2+\nu_3)^2(\nu_1+\nu_2+\nu_3)^2(\nu_1+2\nu_2+\nu_3)^2(2\nu_1+2\nu_2+\nu_3)^2,$$
where
\begin{equation}
\label{xi5}
\Xi_5=\{(\nu_1,\nu_2,\nu_3)\in\mathbb{N}^3\,\mid\,(2\nu_1+2\nu_2+\nu_3)^2+(2\nu_2+\nu_3)^2+\nu_3^2=35-40\gamma\lambda,\,\,\nu_3\equiv 1(\Mod 2)\}.
\end{equation}

The least by modulus non-zero eigenvalue of Laplacian is equal to $-\frac{3}{5\gamma}$ and corresponds to irreducible complex representation of the Lie group $SO(7)$ with highest weight
$\bar{\omega}_1$. The dimension of this representation is equal to $7$. Consequently the multiplicity of the eigenvalue $-\frac{3}{5\gamma}$ is equal to $7^2=49$.

\section{Calculation of Laplacian spectrum for Lie groups $\Sp(3)$ and $\Sp(3)/C(\Sp(3))$}\label{Sec:C3}

\begin{remark}
Remind that $\Sp(2n,C)$ is the Lie group of complex matrices $g$ of order $2n$ with determinant $1$, satisfying condition
$$g^{T}\left(\begin{array}{cc}
0 & E_n \\
-E_n & 0
\end{array}\right)g=\left(\begin{array}{cc}
0 & E_n \\
-E_n & 0
\end{array}\right),$$
where $E_n$ is unit matrix of $n$th order. It is known that the Lie group $\Sp(n)$ is isomorphic to the Lie group $U(2n)\cap\Sp(2n,C)$.  Therefore $C(\Sp(n))$ is isomorphic to
$(\pm E_{2n})\subset U(2n)\cap\Sp(2n,C)$, while   $\Sp(n)/C(\Sp(n))$
is isomorphic to the Lie group $(U(2n)\cap\Sp(2n,C) )/(\pm E_{2n})$.
\end{remark}

By table \ref{Tab:list}, the Lie groups under considerations correspond to root system $C_3$. We apply table~III from \cite{Burb}. Simple roots are
$\alpha_1=\varepsilon_1-\varepsilon_2,\,\,\alpha_2=\varepsilon_2-\varepsilon_3,\,\,\alpha_3=2\varepsilon_3$; maximal root is
$\tilde{\alpha}=2\alpha_1+2\alpha_2+\alpha_3=2\varepsilon_1$.
Positive roots are $2\varepsilon_i$,  $i=1,2,3$, $\varepsilon_1\pm\varepsilon_2$, $\varepsilon_1\pm\varepsilon_3$, $\varepsilon_2\pm\varepsilon_3$.
The sum of positive roots is equal to
$2\beta=6\varepsilon_1+4\varepsilon_2+2\varepsilon_3$, whence
\begin{equation}
\label{p1}
\beta=3\varepsilon_1+2\varepsilon_2+\varepsilon_3,\quad \tilde{\alpha}+\beta=5\varepsilon_1+2\varepsilon_2+\varepsilon_3.
\end{equation}

We act according to algorithm presented in theorem \ref{alg}.

1) $b=\langle\tilde{\alpha}+\beta,\tilde{\alpha}+\beta\rangle-\langle\beta,\beta\rangle=30-14=16.$

2) $\left(\cdot,\cdot\right)=\frac{1}{16}\langle\cdot,\cdot\rangle.$

3) Fundamental weights have the following form
$$\bar{\omega}_1=\varepsilon_1,\quad
\bar{\omega}_2=\varepsilon_1+\varepsilon_2,\quad
\bar{\omega}_3=\varepsilon_1+\varepsilon_2+\varepsilon_3.$$
It is easy to see that
$$\tilde{\alpha}=2\bar{\omega}_1,\quad\beta=\bar{\omega}_1+\bar{\omega}_2+\bar{\omega}_3.$$

4) Set $\Lambda=\sum\limits_{i=1}^{3}\Lambda_i\bar{\omega}_i\in\Lambda^{+}(\mathfrak{sp}(3)),$ where $\Lambda_i\in\mathbb{Z}$, $\Lambda_i\geq 0$, $i=1,2,3$. Then
\begin{equation}
\label{p2}
\Lambda+\beta=\sum_{i=1}^{3}(\Lambda_i+1)\bar{\omega}_i=
\sum_{i=1}^{3}\nu_i\bar{\omega}_i
=\left(\nu_1+\nu_2+\nu_3\right)\varepsilon_1+
\left(\nu_2+\nu_3\right)\varepsilon_2+\nu_3\varepsilon_3,
\end{equation}
where $\nu_i=\Lambda_i+1$, $\nu_i\in\mathbb{N}$, $i=1,2,3.$

By formula (\ref{lam}), eigenvalue $\lambda(\Lambda),$ corresponding to highest weight
$\Lambda,$ is equal to
\begin{equation}
\label{c1}
\lambda(\Lambda)=-\frac{1}{16\gamma}[\langle\Lambda+\beta,\Lambda+\beta\rangle-\langle\beta,\beta\rangle]=
-\frac{1}{16\gamma}\left((\nu_1+\nu_2+\nu_3)^2+(\nu_2+\nu_3)^2+\nu_3^2-14\right).
\end{equation}

By formula (\ref{dim}), dimension $d(\Lambda+\beta)$ of representation
$\rho(\Lambda),$ associated with highest weight $\Lambda,$ is equal to
$$d(\Lambda+\beta)=\frac{(\Lambda+\beta,2\varepsilon_1)}{(\beta,2\varepsilon_1)}\cdot\frac{(\Lambda+\beta,2\varepsilon_2)}{(\beta,2\varepsilon_2)}\cdot
\frac{(\Lambda+\beta,2\varepsilon_3)}{(\beta,2\varepsilon_3)}\cdot\frac{(\Lambda+\beta,\varepsilon_1-\varepsilon_2)}{(\beta,\varepsilon_1-\varepsilon_2)}\times$$
$$\times\frac{(\Lambda+\beta,\varepsilon_1-\varepsilon_3)}{(\beta,\varepsilon_1-\varepsilon_3)}\cdot
\frac{(\Lambda+\beta,\varepsilon_2-\varepsilon_3)}{(\beta,\varepsilon_2-\varepsilon_3)}\cdot
\frac{(\Lambda+\beta,\varepsilon_1+\varepsilon_2)}{(\beta,\varepsilon_1+\varepsilon_2)}
\cdot\frac{(\Lambda+\beta,\varepsilon_1+\varepsilon_3)}{(\beta,\varepsilon_1+\varepsilon_3)}
\cdot\frac{(\Lambda+\beta,\varepsilon_2+\varepsilon_3)}{(\beta,\varepsilon_2+\varepsilon_3)}.$$

On the ground of (\ref{p1}) and (\ref{p2}),
$$\frac{(\Lambda+\beta,2\varepsilon_1)}{(\beta,2\varepsilon_1)}=
\frac{\nu_1+\nu_2+\nu_3}{3},\quad
\frac{(\Lambda+\beta,2\varepsilon_2)}{(\beta,2\varepsilon_2)}=
\frac{\nu_2+\nu_3}{2},$$
$$\frac{(\Lambda+\beta,2\varepsilon_3)}{(\beta,2\varepsilon_3)}=\nu_3,\quad
\frac{(\Lambda+\beta,\varepsilon_1-\varepsilon_2)}{(\beta,\varepsilon_1-\varepsilon_2)}=
\nu_1,\quad
\frac{(\Lambda+\beta,\varepsilon_1-\varepsilon_3)}{(\beta,\varepsilon_1-\varepsilon_3)}=
\frac{\nu_1+\nu_2}{2},$$
$$\frac{(\Lambda+\beta,\varepsilon_2-\varepsilon_3)}{(\beta,\varepsilon_2-\varepsilon_3)}=
\nu_2,\quad
\frac{(\Lambda+\beta,\varepsilon_1+\varepsilon_2)}{(\beta,\varepsilon_1+\varepsilon_2)}=
\frac{\nu_1+2\nu_2+2\nu_3}{5},$$
$$\frac{(\Lambda+\beta,\varepsilon_1+\varepsilon_3)}{(\beta,\varepsilon_1+\varepsilon_3)}=
\frac{\nu_1+\nu_2+2\nu_3}{4},\quad
\frac{(\Lambda+\beta,\varepsilon_2+\varepsilon_3)}{(\beta,\varepsilon_2+\varepsilon_3)}=
\frac{\nu_2+2\nu_3}{3}.$$
Consequently
\begin{equation}
\label{c2}
d(\Lambda+\beta)=\frac{1}{720}\nu_1\nu_2\nu_3(\nu_1+\nu_2)(\nu_2+\nu_3)(\nu_2+2\nu_3)(\nu_1+\nu_2+\nu_3)(\nu_1+\nu_2+2\nu_3)(\nu_1+2\nu_2+2\nu_3).
\end{equation}

\begin{remark}
It follows from (\ref{c1}) and (\ref{c2}) that
$\lambda(\Lambda)=0 \Leftrightarrow \Lambda = 0\Leftrightarrow\nu_i=1$, $i=1,2,3$ and hence the multiplicity $\sigma(\lambda(0))=d^2(0+\beta)=1$ what corresponds to formula (\ref{Eq:spec}).
Moreover, on the base of proposition \ref{Prop:spec_ad}, highest weight of the adjoined representation $\Ad$ is the maximal root
$\tilde{\alpha}=2\alpha_1+2\alpha_2+\alpha_3=2\bar{\omega}_1$ with $\nu_1 = 3 $ and
$\nu_2 = \nu_3 = 1$. Therefore eigenvalue and dimension of representation $\Ad$ can be calculated as follows
$$\lambda_{\Ad}=\lambda(\tilde{\alpha})= -\frac{1}{16\gamma}\left(25+4+1-14\right)= -\frac{1}{\gamma},$$
$$\dim\Ad=d(\tilde{\alpha}+\beta)=\frac{3\cdot 4\cdot 2\cdot 3\cdot 5\cdot 6\cdot 7}{720}=21,$$
what corresponds to formula (\ref{Eq:spec_ad}) for $\gamma=1$ with taking into account the equality  $\dim \Sp(3)=21$ from Table \ref{Tab:list}.
\end{remark}

5) Simple roots and fundamental weights of the Lie algebra $\mathfrak{sp}(3),$ indicated above, define respective lattices $\Lambda_0(C_3)$ and $\Lambda_1(C_3)$ in the following manner
\begin{equation}
\label{c3}
\Lambda_0(C_3)=\left\{\sum\limits_{i=1}^{3}\Psi_i\alpha_i\,\mid\,\Psi_1,\Psi_2,\Psi_3\in\mathbb{Z}\right\},\quad
\Lambda_1(C_3)=\left\{\sum\limits_{i=1}^{3}\Lambda_i\bar{\omega}_i\,\mid\,\Lambda_1,\Lambda_2,\Lambda_3\in\mathbb{Z}\right\}.
\end{equation}

After expressing roots via fundamental weights
$$\alpha_1=2\bar{\omega}_1-\bar{\omega}_2,\quad \alpha_2=-\bar{\omega}_1+2\bar{\omega}_2-\bar{\omega}_3,\quad \alpha_3=-2\bar{\omega}_2+2\bar{\omega}_3$$
and the change of variables
$$\left\{
\begin{array}{l}
\Psi_1=2\Omega_1+\Omega_2+2\Omega_3, \\
\Psi_2=2\Omega_1+2\Omega_2+3\Omega_3, \\
\Psi_3=\Omega_1+\Omega_2+2\Omega_3
\end{array}\right.\quad\Leftrightarrow\quad
\left\{
\begin{array}{l}
\Omega_1=\Psi_1-\Psi_3, \\
\Omega_2=-\Psi_1+2\Psi_2-2\Psi_3, \\
\Omega_3=-\Psi_2+2\Psi_3
\end{array}\right.$$
the lattice takes the following form
\begin{equation}
\label{c4}
\Lambda_0(C_3)=\{2\Omega_1\bar{\omega}_1+\Omega_2\bar{\omega}_2+\Omega_3(\bar{\omega}_1+\bar{\omega}_3)\,\,\mid\,\Omega_1,\Omega_2,\Omega_3\in\mathbb{Z}\}.
\end{equation}
Let us define lattice $\Lambda_1(C_3)$ in base
$\{\bar{\omega}_1,\bar{\omega}_2,\bar{\omega}_1+\bar{\omega}_3\}$ via the change of variables
$\{\Lambda_1=\Omega_1+\Omega_3,\,\Lambda_2=\Omega_2,\,\Lambda_3=\Omega_3\}$:
\begin{equation}
\label{c5}
\Lambda_1(C_3)=\{\Omega_1\bar{\omega}_1+\Omega_2\bar{\omega}_2+\Omega_3(\bar{\omega}_1+\bar{\omega}_3)\,\,\mid\,\Omega_1,\Omega_2,\Omega_3\in\mathbb{Z}\}.
\end{equation}

It follows from (\ref{c4}) and (\ref{c5}) that $\Lambda_0(C_3)\subset\Lambda_1(C_3)$ and
$\Lambda_1(C_3)/\Lambda_0(C_3)\cong\mathbb{Z}_2$, i.e. it has prime order. It follows from here and corollary \ref{Col:LCalc} that there is no other lattices. Lie groups, associated with these characteristic lattices, are presented in table \ref{Tab:list}. The sets of highest weights, corresponding to these lattices, are depicted in figure \ref{fig:lat_c} above.
\begin{figure}
\includegraphics[width=0.55\textwidth]{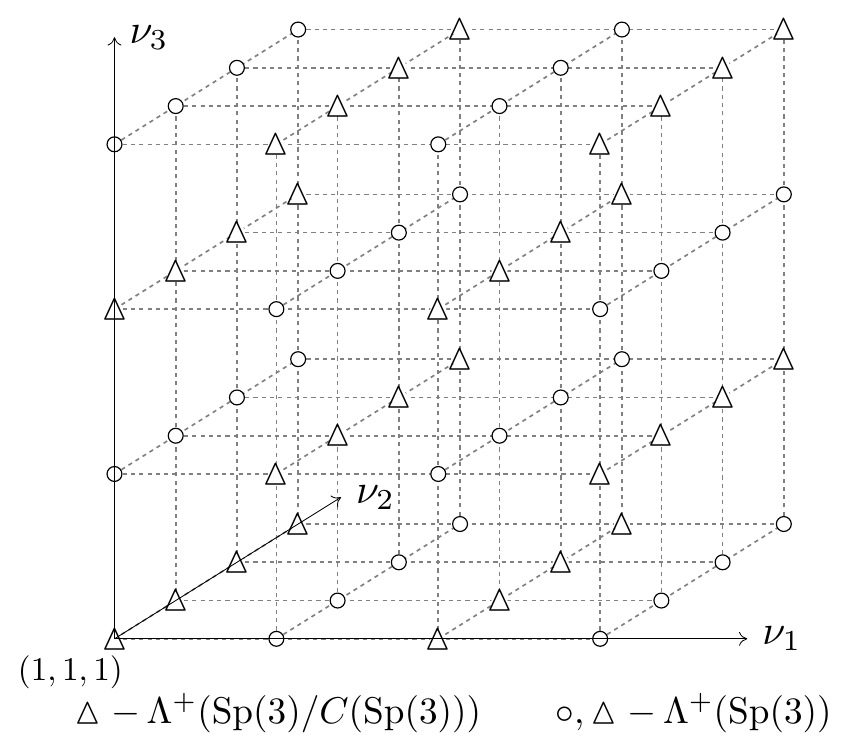}
\caption{The sets of highest weights of Lie groups, corresponding to the root system  $C_3$}
\label{fig:lat_c}
\end{figure}

a) Let us give formulas, defining Laplacian spectrum of the Lie group $\Sp(3)$.

6a) By formula (\ref{c3}), the set of highest weights of the Lie group $\Sp(3)$ is equal to
$$\Lambda^{+}(\Sp(3))=\Lambda^{+}_1(C_3)=\left\{\sum\limits_{i=1}^{3}\Lambda_i\bar{\omega}_i\,\Bigm|\,\Lambda_i\in\mathbb{Z},\,\,\Lambda_i\geq 0,\,\,i=1,2,3\right\}.$$

7a) Set $\Lambda=\sum\limits_{i=1}^{3}\Lambda_i\bar{\omega}_i\in\Lambda^{+}(\Sp(3))$, then by p.~4) eigenvalue $\lambda(\Lambda)$ and dimension $d(\Lambda+\beta)$ are calculated respectively by formulas (\ref{c1}) and (\ref{c2}), where $\nu_i=\Lambda_i+1$, $\nu_i\in\mathbb{N}$, $i=1,2,3.$

8a) Applying formula (\ref{abc}) and result of preceding point, we get the following multiplicity of eigenvalue $\lambda(\Lambda)$
$$\sigma(\Lambda)=\frac{1}{(720)^2}\sum_{\Xi_6}
\nu_1^2\nu_2^2\nu_3^2(\nu_1+\nu_2)^2(\nu_2+\nu_3)^2(\nu_2+2\nu_3)^2(\nu_1+\nu_2+\nu_3)^2(\nu_1+\nu_2+2\nu_3)^2(\nu_1+2\nu_2+2\nu_3)^2,$$
where
\begin{equation}
\label{xi6}
\Xi_6=\{(\nu_1,\nu_2,\nu_3)\in\mathbb{N}^3\,\mid\,(\nu_1+\nu_2+\nu_3)^2+(\nu_2+\nu_3)^2+\nu_3^2=14-16\gamma\lambda\}.
\end{equation}

The least by modulus non-zero eigenvalue of Laplacian is equal to $-\frac{7}{16\gamma}$ and corresponds to irreducible complex representation with highest weight $\bar{\omega}_1$. The dimension of this representation of the Lie group $\Sp(3)$ is equal to $6$. Therefore the multiplicity of eigenvalue
$-\frac{7}{16\gamma}$ is equal to $6^2=36$.

b) Let present formulas giving Laplacian spectrum of the Lie group $\Sp(3)/C(\Sp(3))$.

6b) By formula (\ref{c4}), the  set of highest weights of the Lie group is equal to
$\Sp(3)/C(Sp(3))$
\begin{equation*}
\begin{array}{r c l}
&&\Lambda^{+}(\Sp(3)/C(\Sp(3)))=\Lambda_0(C_3)\cap\Lambda^{+}_1(C_3)=\\
&&\{(2\Omega_1+\Omega_3)\bar{\omega}_1+\Omega_2\bar{\omega}_2+\Omega_3\bar{\omega}_3\mid\Omega_1,\Omega_2,\Omega_3\in\mathbb{Z},\,\Omega_2\geq 0,\,\Omega_3\geq 0,\,2\Omega_1+\Omega_3\geq 0\}.
\end{array}
\end{equation*}

7b) Set $\Lambda=(2\Omega_1+\Omega_3)\bar{\omega}_1+\Omega_2\bar{\omega}_2+\Omega_3\bar{\omega}_3\in\Lambda^{+}(\Sp(3)/C(\Sp(3)))$, then by p.~4) eigenvalue $\lambda(\Lambda)$ and dimension  $d(\Lambda+\beta)$ are calculated respectively by formulas (\ref{c1}) and (\ref{c2}), where
$\nu_1=2\Omega_1+\Omega_3+1$, $\nu_2=\Omega_2+1$, $\nu_3=\Omega_3+1$, $\nu_i\in\mathbb{N}$, $i=1,2,3,$ $\nu_1\equiv \nu_3(\Mod 2)$.

8b) Applying formula (\ref{abc}) and results of preceding point, we get the following multiplicity of eigenvalue $\lambda(\Lambda)$
$$\sigma(\Lambda)=\frac{1}{(720)^2}\sum_{\Xi_7}
\nu_1^2\nu_2^2\nu_3^2(\nu_1+\nu_2)^2(\nu_2+\nu_3)^2(\nu_2+2\nu_3)^2(\nu_1+\nu_2+\nu_3)^2(\nu_1+\nu_2+2\nu_3)^2(\nu_1+2\nu_2+2\nu_3)^2,$$
where
\begin{equation}
\label{xi7}
\Xi_7=\{(\nu_1,\nu_2,\nu_3)\in\mathbb{N}^3\,\mid\,(\nu_1+\nu_2+\nu_3)^2+(\nu_2+\nu_3)^2+\nu_3^2=14-16\gamma\lambda,\,\nu_1\equiv\nu_3(\Mod 2)\}.
\end{equation}

The least by modulus non-zero eigenvalue of Laplacian is equal to $-\frac{3}{4\gamma}$ and corresponds to irreducible complex representation of the Lie group $\Sp(3)/C(\Sp(3))$ with highest weight $\bar{\omega}_2$. The dimension of this representation is equal to $14$. Consequently the multiplicity of the eigenvalue $-\frac{3}{4\gamma}$ is equal to $14^2=196$.

\section{Necessary information from number theory}

In this section are presented all necessary information on classical solutions of presentation problem of natural numbers by values of some positively defined ternary (or binary) integer quadratic forms on integer three-dimensional (respectively, two-dimensional) vectors, applying in next sections.

\begin{theorem}\cite{Ven}, \cite{Dev}.
\label{sq}
A natural number $k$ can be presented in the form
\begin{equation}
\label{square1}
k=x^2+y^2,\quad x,\,y\in\mathbb{Z},
\end{equation}
if and only if $k$ has no prime factor $p$ with condition $p\equiv 3(\Mod 4)$, which occurs in odd
power into factorization of $k$ by prime factors.

In addition a number $N_2(k)$ of all solutions to the equation (\ref{square1}) is equal to quadruplicate difference of quantities of (natural) divisors $d$ of $k$ such that
$d\equiv 1(\Mod 4)$ and divisors $d$ of $k$ such that $d\equiv 3(\Mod 4)$.
\end{theorem}

\begin{theorem}
\label{sq1}
Assume that a natural number $k$ can be presented in the form
\begin{equation}
\label{squared21}
k=x^2+y^2,\quad x<y,\,\,x,\,y\in\mathbb{N},
\end{equation}
$L_2(k)$ is the quantity of such presentations. Then

1. If $k\neq m^2$, $k\neq 2m^2$ for any $m\in\mathbb{N}$, then $L_2(k)=N_2(k)/8$.

2. If $k=m^2$ or $k=2m^2$ for some $m\in\mathbb{N}$, then $L_2(k)=\left(N_2(k)-4\right)/8$.

\end{theorem}

\begin{proof}
1. Let $k\neq m^2$, $k\neq 2m^2$ for all $m\in\mathbb{N}$. To every presentation $(x,y)$ of the number $k$ in the form (\ref{squared21}) correspond exactly 8 of different presentations of $k$ of the form (\ref{square1}), namely, ordered pairs $(\pm x,\pm y)$, $(\pm y,\pm x)$.
Therefore $N_2(k)=8L_2(k)$, whence follows the required formula.

2. Let $k=m^2$ ($k=2m^2$) for some $m\in\mathbb{N}$. Then, besides presentations of $k$ in the form  (\ref{square1}), described in p.~1, there are also only 4 different presentations of the form (\ref{square1}), namely, ordered pairs $(\pm m,0)$, $(0,\pm m)$ (respectively, $(\pm m,\pm m)$). Therefore $N_2(k)=8L_2(k)+4$, whence follows the required formula.
\end{proof}

Theorems \ref{sq} and \ref{sq1} imply directly

\begin{corollary}
\label{cor}
A natural number $k$ can be presented in the form (\ref{squared21}) if and only if the following conditions are fulfilled.

1. $k\geq 5$ and $k$ has no prime divisor $p$ with condition $p\equiv 3(\Mod 4)$, which occurs in odd power into factorization of $k$ by prime factors.

2. If $k=m^2$ or $k=2m^2$ for some $m\in\mathbb{N}$ then the difference of quantities of (natural) divisors $d$ of $k$ such that $d\equiv 1(\Mod 4)$ and divisors $d$ of $k$ such that
$d\equiv 3(\Mod 4)$ is more than 1.
\end{corollary}

\begin{theorem} \cite{Buh}
An odd natural number $k$ can be presented in the form
\begin{equation}
\label{square2}
k=x^2+2y^2,\quad x,\,y\in\mathbb{Z},
\end{equation}
where $\mbox{GCD}(x,y)=1$, if and only if the factorization of $k$ by prime factors does not contain prime numbers of the form $8n+5$ и $8n+7$.
\end{theorem}

\begin{theorem} \cite{Ven}
For any $k\in\mathbb{N}$ the number $N_2^{\prime}(k)$ of all solutions to the equation (\ref{square2}) is equal to doubled difference of quantities (of natural) divisors $d$ of $k$ such that $d\equiv 1(\Mod 8)$ or $d\equiv 3(\Mod 8)$ and divisors $d$ of $k$ such that
$d\equiv 5(\Mod 8)$ or $d\equiv 7(\Mod 8)$.
\end{theorem}

\begin{theorem}
\cite{E}
If natural numbers $k_1,$ $k_2$ admit presentation of the form (\ref{square1}) (respectively, (\ref{square2})), then also the number $k=k_1\cdot k_2$ admits presentation of the form (\ref{square1}) (respectively, (\ref{square2})).
\end{theorem}

\begin{theorem} \cite{Dev}
A natural number $k$ can be presented in the form
\begin{equation}
\label{three}
k=x^2+y^2+z^2,\quad x,\,y,\,z\in\mathbb{Z},
\end{equation}
if and only if $k$ cannot be presented in the form $4^{m}(8l+7)$, where $m,$ $l\in\mathbb{Z}_{+}$.
\end{theorem}

\begin{proposition}\label{AAA}
1. If a natural number $k$ can be presented in the form (\ref{square1}), where $x\equiv y\equiv 1(\Mod 2)$, then $k\equiv 2(\Mod 8)$.
Conversely, if $k\equiv 2(\Mod 8)$ and $k$ can be presented in the form (\ref{square1}), then $x\equiv y\equiv 1(\Mod 2)$.

2. If a natural number $k$ can be presented in the form (\ref{three}), where $x\equiv y\equiv z\equiv 1(\Mod 2)$, then $k\equiv 3(\Mod 8)$.
Conversely, if $k\equiv 3(\Mod 8)$ and $k$ can be presented in the form (\ref{three}), then $x\equiv y\equiv z\equiv 1(\Mod 2)$.
\end{proposition}

\begin{proof}
These statements follow from the fact that the square of any odd number gives residue $1$ under division by $8$, while the square of any even number
gives residue $0$ or $4$ under division by $8$.
\end{proof}

Later we shall need the notion of \textit{the Legendre--Jacobi symbol}. There are different definitions of this symbol. Here is given the most simple its definition, taken from \cite{C} and belonging to Russian mathematician E.I.~Zolotarev (1847--1878).

\begin{definition}
Let $n>1$ be an odd natural number, $a$ be an integer number, coprime with $n.$ The Legendre--Jacobi
symbol $\left(\frac{a}{n}\right)$ is the sign of permutation on residue ring $\Mod n$ obtained by multiplication of this ring by $a\Mod n.$
\end{definition}

\begin{theorem} \cite{Ven}
1. Let $k\in\mathbb{N}$, $k=1, 2(\Mod\,\,4)$, $k\neq 1$. Then the number $\psi(k)$ of proper presentations of $k$ in the form (\ref{three}) is finite and equal to $12h(k)$, where
\begin{equation}
\label{h}
h(k)=\sum\left(-\frac{k}{a}\right),
\end{equation}
and the summation is taken for all $a$ such that $a\in\mathbb{N}$, $0<a<k$, $a$ is coprime with
$2k$; $\left(-\frac{k}{a}\right)$ is the Legendre--Jacobi symbol.

2. Let $k\in\mathbb{N}$, $k=3(\Mod\,\,8)$, $k\neq 3$. Then the number $\psi(k)$ of proper presentations of $k$ in the form (\ref{three}) is finite and equal to $24h^{\prime}(k)$, where
\begin{equation}
\label{hprime}
h^{\prime}(k)=\frac{1}{3}\sum\left(\frac{b}{k}\right),
\end{equation}
and the summation is taken for all $b$ such that $b\in\mathbb{N}$, $0<b<k$, $b$ is coprime with  $2k$; $\left(\frac{b}{k}\right)$ is the Legendre--Jacobi symbol.
\end{theorem}

Let us define function $F(k)$, $k\in\mathbb{N}$, by the following rule

1) If  $k=(2m-1)^2$ for some $m\in\mathbb{N}$, then
$$F(k)=\sum h(k/\delta^2)-\frac{1}{2};$$

2) If $k\neq (2m-1)^2$ for any $m\in\mathbb{N}$, then
$$F(k)=\sum h(k/\delta^2).$$
In both cases the summation is taken by all square divisors $\delta^2$ of the number $k$.

\begin{theorem} \cite{Ven}
Let $k\in\mathbb{N}$, $N_3(k)$ be a number of all presentations of the number $k$ in the form (\ref{three}). Then the following statements are valid.

1. If $k=1, 2(\Mod\,\,4)$, $k\neq 1$, then $N_3(k)=12F(k)$.

2. If $k=3(\Mod\,\,8)$, $k\neq 3$, then  $N_3(k)=8F(k)$.

3. If $k=7(\Mod\,\,8)$ then $N_3(k)=0$.

4. If $k=0(\Mod\,\,4)$ then $N_3(k)=N_3(k/4)$.
\end{theorem}

\begin{theorem}
\label{3sq}
Assume that a natural number $k$ can be presented in the form
\begin{equation}
\label{squared3}
k=x^2+y^2+z^2,\quad x<y<z,\,\,x,\,y,\,z\in\mathbb{N},
\end{equation}
$L_3(k)$ is the quantity of such presentations. Then

1. If $k\neq\alpha m^2$ for any $m\in\mathbb{N}$, $\alpha=1,2,3$, then
$L_3(k)=\left(N_3(k)-3N_2(k)-6N_2^{\prime}(k)\right)/48.$

2. If $k=m^2$ for some $m\in\mathbb{N}$ then
$L_3(k)=\left(N_3(k)-3N_2(k)-6N_2^{\prime}(k)+18\right)/48.$

3. If $k=2m^2$ for some $m\in\mathbb{N}$ then
$L_3(k)=\left(N_3(k)-3N_2(k)-6N_2^{\prime}(k)+12\right)/48.$

4. If $k=3m^2$ for some $m\in\mathbb{N}$ then
$L_3(k)=\left(N_3(k)-6N_2^{\prime}(k)+16\right)/48.$
\end{theorem}

\begin{proof}
1. Let $k\neq\alpha m^2$ for any $m\in\mathbb{N}$, $\alpha=1,2,3$. Then $k$ has $L_3(k)$ presentations of the form (\ref{squared3}), $N_3(k)$ presentations of the form (\ref{three}), $N_2(k)$ presentations of the form (\ref{square1}), and $N_2^{\prime}(k)$ presentations of the form  (\ref{square2}).

To every presentation $(x,y,z)$ of $k$ of the form (\ref{squared3}) correspond exactly 48 different presentation of $k$ of the form (\ref{three}), namely, ordered triples $(\pm x,\pm y, \pm z)$,
$(\pm y,\pm x,\pm z)$, $(\pm x,\pm z,\pm y)$, $(\pm y,\pm z,\pm x)$, $(\pm z,\pm x,\pm y)$,
$(\pm z,\pm y,\pm x)$.

If $(x,y)$ is a presentation of $k$ of the form (\ref{square1}) then it follows from conditions of p.~1, imposed on $k$, that $x\neq y$, $x\neq 0$, $y\neq 0$. Therefore to every octuple of different
presentations $(\pm x,\pm y)$, $(\pm y,\pm x)$ of $k$ in the form (\ref{square1}) correspond exactly 24 different presentations of $k$ in the form (\ref{three}), namely, ordered triples
$(0,\pm x,\pm y)$, $(\pm x,0,\pm y)$, $(\pm x,\pm y,0)$, $(0,\pm y,\pm x)$, $(\pm y,0,\pm x)$,
$(\pm y,\pm x,0)$.

If $(x,y)$ is a presentation of $k$ in the form (\ref{square2}) then it follows from conditions in
p.~1, imposed on $k$, that $x\neq y$, $x\neq 0$, $y\neq 0$. Therefore to any quadruple of different presentations $(\pm x,\pm y)$ of $k$ in the form (\ref{square2}) correspond exactly 24 different presentations of $k$ in the form (\ref{three}), namely, ordered triples  $(\pm x,\pm y,\pm y$,
 $(\pm y,\pm x,\pm y)$, $(\pm y,\pm y,\pm x)$.

Thus $N_3(k)=48L_3(k)+3N_2(k)+6N^{\prime}_2(k)$, whence it follows the required formula.

2. Let $k=m^2$ for some $m\in\mathbb{N}$. Then, besides presentations of $k$ in the form  (\ref{three}), described in p.~1, there are also only 6 different presentations of $k$ in the form (\ref{three}), namely, ordered triples $(0,\pm m,0)$, $(\pm m,0,0)$, $(0,0,\pm m)$. Besides presentations $(x,y)$, $x\neq y$, $x\neq 0$, $y\neq 0$ of $k$ in the form (\ref{square1}) there are also 4 of different presentations of $k$ in the form (\ref{square1}), namely, ordered pairs
$(0,\pm m)$, $(\pm m,0)$. Except presentations $(x,y)$, $x\neq y$, $x\neq 0$, $y\neq 0$ of $k$ in the form (\ref{square2}) there are also 2 different presentations of $k$ in the form  (\ref{square2}), namely, ordered pairs $(\pm m,0)$. Then on the ground of p.~1
$$N_3(k)-6=48L_3(k)+3(N_2(k)-4)+6(N^{\prime}_2(k)-2),$$
whence it follows the required formula.

3. Let $k=2m^2$ for some $m\in\mathbb{N}$. Then, besides presentations of $k$ in the form (\ref{three}), described in p.~1, there are also only 12 different presentations of $k$ in the form  (\ref{three}), namely, ordered triples $(0,\pm m,\pm m)$, $(\pm m,0,\pm m)$, $(\pm m,\pm m,0)$. Except presentations  $(x,y)$, $x\neq y$, $x\neq 0$, $y\neq 0$ of $k$ in the form (\ref{square1}) there exist also 4 different presentations of $k$ in the form (\ref{square1}), namely, ordered pairs $(\pm m,\pm m)$. Besides presentations $(x,y)$ of $k$ with conditions $x\neq y$, $x\neq 0$,
$y\neq 0$ of the form (\ref{square2}) there exist yet 2 different presentations of $k$ in the form
(\ref{square2}), namely, ordered pairs $(0,\pm m)$. Then on the base of p.~1
$$N_3(k)-12=48L_3(k)+3(N_2(k)-4)+6(N^{\prime}_2(k)-2),$$
whence it follows the required formula.

4. Let $k=3m^2$ for some $m\in\mathbb{N}$. Then, besides presentation of $k$ in the form (\ref{three}), described in p.~1, there are also 8 different presentations of $k$ in the form
(\ref{three}), namely, ordered triples $(\pm m,\pm m,\pm m)$. On the ground of theorem \ref{sq}, $N_2(3m^2)=0$. Except presentations $(x,y)$ of $k$ with conditions $x\neq y$, $x\neq 0$, $y\neq 0$, in the form (\ref{square2}) there exist yet 4 different presentations in the form (\ref{square2}), namely, ordered pairs $(\pm m,\pm m)$. Then on the base of p.~1
$$N_3(k)-8=48L_3(k)+6(N^{\prime}_2(k)-4),$$
whence it follows the required formula.
\end{proof}

Let $\Gamma\subset \mathbb{R}^n$ be a lattice, i.e. additive discrete subgroup such that the quotient group $\mathbb{R}^n/\Gamma$ is compact; $N_{\Gamma}(m)$ is the number of vectors
$x\in \Gamma$ with the square of norm $m,$ i.e. such that $x\cdot x=m.$ The calculation of the number $N_{\Gamma}(m)$ is facilitated by introducing \textit{theta-series} (or
\textit{theta-function}) of the lattice $\Gamma$ \cite{CS} which is defined as
\begin{equation}
\label{theta}
\Theta_{\Gamma}(q)=\sum_{x\in \Gamma}q^{x\cdot x}=\sum_{m=0}^{\infty}N_{\Gamma}(m)q^{m}.
\end{equation}

For example, theta-series of integer lattice $\mathbb{Z}$ is equal to
\begin{equation}
\label{zeta}
\Theta_{\mathbb{Z}}(q)=\sum_{z\in \mathbb{Z}}q^{z^2}=1+2q+2q^4+2q^9+2q^{16}+ \dots,
\end{equation}
\begin{equation}
\label{zetak}
\Theta_{\sqrt{2}\mathbb{Z}}(q)=\sum_{z\in \mathbb{Z}}q^{2z^2}=1+2q^2+2q^8+2q^{18}+2q^{32}+ \dots.
\end{equation}
It is easy to see that if $\Gamma=\Gamma_1\oplus \Gamma_2$ is direct orthogonal sum of lattices, then
\begin{equation}
\label{suma}
\Theta_{\Gamma}(q)=\Theta_{\Gamma_1}(q)\cdot \Theta_{\Gamma_2}(q).
\end{equation}
In particular,
\begin{equation}
\label{zetan}
\Theta_{\mathbb{Z}^n}(q)= \Theta_{\mathbb{Z}}(q)^n,
\end{equation}
\begin{equation}
\label{zetaz}
\Theta_{\mathbb{Z}\oplus \sqrt{2}\mathbb{Z}}(q)=\Theta_{\mathbb{Z}}(q)\cdot \Theta_{\sqrt{2}\mathbb{Z}}(q).
\end{equation}
In previous notation, $N_{\mathbb{Z}^2}(k)=N_2(k),$ $N_{\mathbb{Z}^3}(k)=N_3(k),$
$N_{\mathbb{Z}\oplus \sqrt{2}\mathbb{Z}}(k)=N'_2(k).$ Formulas (\ref{zeta}) -- (\ref{zetaz}) permit to calculate these numbers. So,
$$
\Theta_{\mathbb{Z}^2}(q)= \Theta_{\mathbb{Z}}(q)^2=(1+2q+2q^4+2q^9+2q^{16}+ \dots)(1+2q+2q^4+2q^9+2q^{16}+ \dots)=
$$
$$1+4q+ 4q^2+ 4q^4+ 8q^5+ 4q^8+ 4q^9+ 8q^{10}+ 8q^{13}+ 4q^{16}+ 8q^{17}+ 4q^{18}+ 8q^{20}+ \dots,$$

$$
\Theta_{\mathbb{Z}^3}(q)= \Theta_{\mathbb{Z}}(q)^3=(1+2q+2q^4+2q^9+2q^{16}+ \dots)\times
$$
$$(1+4q+ 4q^2+ 4q^4+ 8q^5+ 4q^8+ 4q^9+ 8q^{10}+ 8q^{13}+ 4q^{16}+ 8q^{17}+ 4q^{18}+ 8q^{20}+ \dots)=$$
$$1+6q+12q^2+8q^3+6q^4+24q^5+24q^6+12q^8+30q^9+24q^{10}+24q^{11}+8q^{12}+$$
$$24q^{13}+48q^{14}+6q^{16}+48q^{17}+36q^{18}+24q^{19}+
24q^{20}+48q^{21}+24q^{22}+24q^{24}+\dots,$$

$$\Theta_{\mathbb{Z}\oplus \sqrt{2}\mathbb{Z}}(q)=(1+2q+2q^4+2q^9+2q^{16}+ \dots)(1+2q^2+2q^8+2q^{18}+ \dots)=$$
$$1+2q+2q^2+4q^3+2q^4+4q^6+2q^8+6q^9+4q^{11}+4q^{12}+2q^{16}+4q^{17}+6q^{18}+4q^{19}+4q^{22}+4q^{24}+\dots.$$

Calculations, using equalities, obtained here, and theorems \ref{sq1}, \ref{3sq} give

\begin{proposition}
\label{LL}
1) $L_2(5)=L_2(10)=L_2(13)=L_2(17)=L_2(20)=1;$ $L_2(k)=0$ for all the rest natural numbers
$k\leq 24.$

2) $L_3(14)=L_3(21)=1;$ $L_3(k)=0$ for all the rest natural numbers $k\leq 24.$
\end{proposition}

\section{Conclusion}

\begin{theorem}
\label{qq}
Let $G=\SU(4)$ is supplied by biinvariant Riemannian metric $\nu$ such that $\nu(e)=-k_{ad}$. A number $\lambda < 0$ is eigenvalue of Laplacian on $(G,\nu)$ in one and only one of the following cases (I or II)

I. 1) $-32\lambda\in\mathbb{N}$, $-32\lambda\equiv 7(\Mod 8)$;

2) natural number $k=20-32\lambda$ is a sum of squares of three mutually different natural numbers.

In addition the number of highest weights $\Lambda$, such that $\lambda(\Lambda)=\lambda$, is equal to $2L_3(20-32\lambda)$.

II. 1) $-8\lambda\in\mathbb{N}$;

2) natural number $k=5-8\lambda$ is a sum of squares of three mutually different natural numbers.

In addition the number of highest weights $\Lambda$, such that $\lambda(\Lambda)=\lambda$, is equal to $2L_3(5-8\lambda)+L_2(5-8\lambda)$.
\end{theorem}

\begin{proof}
Following to (\ref{xi1}), let us consider Diophantine equation
\begin{equation}
\label{s1}
(\nu_1+2\nu_2+\nu_3)^2+(\nu_1-\nu_3)^2+(\nu_1+\nu_3)^2=20-32\lambda,\quad\nu_1,\,\,\nu_2,\,\,\nu_3\in\mathbb{N}.
\end{equation}
It is clear that if the equation (\ref{s1}) is solvable, then $-32\lambda\in\mathbb{N}$.

Assume at first that the equation (\ref{s1}) has a solution $(\nu_1,\,\nu_2,\,\nu_3)$ such that
$\nu_1+\nu_3$ is an odd number. Since $\nu_1$ and $\nu_3$ occur symmetrically in (\ref{s1}), then
we can suppose for definiteness that $\nu_1>\nu_3$. Let us introduce the following notation
$$x=\nu_1-\nu_3,\quad y=\nu_1+\nu_3,\quad z=\nu_1+2\nu_2+\nu_3.$$
Then the equation (\ref{s1}) will written down in the form
\begin{equation}
\label{as2}
x^2+y^2+z^2=20-32\lambda,
\end{equation}
furthermore
\begin{equation}
\label{cond}
x,\,y,\,z\in\mathbb{N},\,\,x\equiv 1(\Mod 2),\,y\equiv 1(\Mod 2),\,z\equiv 1(\Mod 2),\quad x<y<z.
\end{equation}
It is easy to see that the number of solutions $(\nu_1,\,\nu_2,\,\nu_3)\in\mathbb{N}^3$ to equation (\ref{s1}), satisfying condition $\nu_1+\nu_3\equiv 1(\Mod 2)$, is equal to doubled number of solutions to equation (\ref{as2}) with condition (\ref{cond}).
Therefore, on the ground of proposition \ref{AAA}, if the equation (\ref{as2}) is solvable under
condition (\ref{cond}), then $20-32\lambda\equiv 3(\Mod 8)$ what is equivalent to relation  $-32\lambda\equiv 7(\Mod 8)$.

Thus, if the equation (\ref{s1}) is solvable and at least one its solution satisfies condition
$\nu_1+\nu_3\equiv 1(\Mod 2)$, then $-32\lambda\equiv 7(\Mod 8)$ and the number of solutions, satisfying next to the last condition, is equal to $2L_3(20-32\lambda)$.

Let us suppose now that the equation (\ref{s1}) has a solution $(\nu_1,\,\nu_2,\,\nu_3)$, where
$\nu_1\neq\nu_3$ are natural numbers of the same evenness. Then $-8\lambda\in\mathbb{Z}_{+}$. Let us set for definiteness that $\nu_1>\nu_3$ and introduce the following notation
\begin{equation}
\label{not}
x=\frac{\nu_1-\nu_3}{2},\quad y=\frac{\nu_1+\nu_3}{2},\quad z=\frac{\nu_1+\nu_3}{2}+\nu_2.
\end{equation}
Then the equation (\ref{s1}) will written down in the form
\begin{equation}
\label{s3}
x^2+y^2+z^2=5-8\lambda,
\end{equation}
in addition
\begin{equation}
\label{cond1}
x,\,y,\,z\in\mathbb{N},\,\, x<y<z.
\end{equation}
Therefore the number of solutions $(\nu_1,\,\nu_2,\,\nu_3)\in\mathbb{N}^3$ of equation (\ref{s1}), satisfying the condition that the numbers $\nu_1\neq\nu_3$ have the same evenness, is equal to  $2L_3(5-8\lambda)$.

Let suppose at the end that the equation (\ref{s1}) has a solution $(\nu_1,\,\nu_2,\,\nu_3)$, where $\nu_1=\nu_3$. After the change of variables $x=\nu_1$, $y=\nu_1+\nu_2$ the equation (\ref{s1}) will written down in the form
\begin{equation}
\label{s4}
x^2+y^2=5-8\lambda,
\end{equation}
in addition $x,\,y\in\mathbb{N},\,\, x<y.$ Therefore the number of solutions
$(\nu_1,\,\nu_2,\,\nu_3)\in\mathbb{N}^3$ to equation (\ref{s1}), satisfying condition
$\nu_1=\nu_3$, is equal to $L_2(5-8\lambda)$.
\end{proof}

\begin{remark}
On the ground of theorem \ref{sq} and corollary \ref{cor}, $L_2(k)=0$ for $k\in N$ if and only if it is satisfied one of the following conditions

a) the number $k$ has at least one prime divisor $p$ with condition $p\equiv 3(\Mod 4)$, appearing in decomposition of $k$ in an odd power;

b) $k=m^2$ or $k=2m^2$ for some $m\in\mathbb{N}$ and the difference of quantities of (natural) divisors $d$ of $k$ such that $d\equiv 1(\Mod 4)$ and divisors $d$ of $k$ such that
$d\equiv 3(\Mod 4),$ is equal to 1.
\end{remark}

\begin{example}
1) If $-32\lambda=15,$ then $-32\lambda\equiv 7(\Mod 8),$ $k=20-32\lambda=35=1^2+3^2+5^2.$ Moreover $\lambda=-15/32$ is the least by modulus non-zero eigenvalue of Laplacian, considered earlier, and the number of highest weights $\Lambda,$ such that $\lambda(\Lambda)=-15/32,$ is equal to $2L_3(35)=2.$

2) If $-8\lambda=9,$ then $k=5-8\lambda=14=1^2+2^2+3^2.$ Moreover $\lambda=-9/8,$
$\Lambda=2\overline{\omega}_1$ and the number of highest weights $\Lambda',$ such that
$\lambda(\Lambda')=-9/8,$ is equal to $2L_3(14)+L_2(14)=2.$

3) If $-8\lambda=16,$ then $5-8\lambda=21=1^2+2^2+4^2.$ In addition $\lambda=-2,$
$\Lambda=2\overline{\omega}_1+\overline{\omega}_2$ and the number of highest weights $\Lambda',$ such that $\lambda(\Lambda')=-2,$ is equal to $2L_3(21)+ L_2(21)=2.$
\end{example}

\begin{theorem}
\label{oo}
Suppose that $G=\SU(4)/(\pm E_4)$, where $E_4$ is unit matrix of the fourth order, is supplied with biinvariant Riemannian metric $\nu$ such that $\nu(e)=-k_{ad}$. A number $\lambda < 0$ is an eigenvalue of Laplacian on $(G,\nu)$ if and only if the following conditions are satisfied

1) $-8\lambda\in\mathbb{N}$;

2) natural number $k=5-8\lambda$ is a sum of squares of three mutually different natural numbers.

In addition the number of highest weights $\Lambda$, such that $\lambda(\Lambda)=\lambda$, is equal to $2L_3(5-8\lambda)+L_2(5-8\lambda)$.
\end{theorem}

\begin{proof}
It follows from (\ref{xi2}) and proof of theorem \ref{qq}.
\end{proof}

\begin{theorem}
\label{qq3}
Assume that $G=\SU(4)/C(\SU(4))$ is supplied by biinvariant Riemannian metric $\nu$ such that
$\nu(e)=-k_{ad}$. A number $\lambda < 0$ is an eigenvalue of Laplacian on $(G,\nu)$ if and only if
the following conditions are satisfied

1) $-8\lambda\in\mathbb{N}$, in addition either $-4\lambda\equiv 3(\Mod 4)$, or $-4\lambda\equiv 2(\Mod 4)$, or $-\lambda\in\mathbb{N}$;

2) natural number $k=5-8\lambda$ is a sum of squares of three mutually different numbers.

Moreover the number of highest weights $\Lambda$, such that $\lambda(\Lambda)=\lambda$, is equal to $2L_3(5-8\lambda)+L_2(5-8\lambda)$.
\end{theorem}

\begin{proof}
Following to (\ref{xi3}), consider Diophantine equation
\begin{equation}
\label{st}
(\nu_1+2\nu_2+\nu_3)^2+(\nu_1-\nu_3)^2+(\nu_1+\nu_3)^2=20-32\lambda,\quad\nu_1,\,\nu_2,\,\nu_3\in\mathbb{N},\,\,
\nu_3-\nu_1+2\nu_2\equiv 2(\Mod 4).
\end{equation}

Assume at first that the equation (\ref{st}) has a solution $(\nu_1,\,\nu_2,\,\nu_3)$, where
$\nu_1\neq\nu_3$. Since $\nu_1$ and $\nu_3$ occur symmetrically in (\ref{st}), then we can suppose without loss of generality that $\nu_1>\nu_3$. It is easy to see that after the change of variables (\ref{not}) (\ref{st}) will written down in the form (\ref{s3}), moreover
 \begin{equation}
\label{st2}
x,\,y,\,z\in\mathbb{N},\,\, x<y<z,\,\,x+y+z\equiv 1(\Mod 2).
\end{equation}
Condition $x+y+z\equiv 1(\Mod 2)$ is equivalent to the fact that either all numbers $x$, $y$, $z$ are odd or there is only one odd number among them.
It follows from the proof of proposition \ref{AAA} that if the equation (\ref{s3}) with condition (\ref{st2}) is solvable, then natural number $5-8\lambda$ can give only residue $1$, $3$ or $5$
under division by $8$, what is equivalent to $-4\lambda\equiv 2(\Mod 4)$,
$-4\lambda\equiv 3(\Mod 4)$ or $-\lambda\in\mathbb{N}$ respectively. It is easy to see that in this case the oddness condition for the quantity of odd numbers among $x$, $y$, $z$ in decomposition (\ref{s3}) is satisfied automatically, and the number of solutions to equation (\ref{s3}) with
condition (\ref{st2}) is equal to $L_3(5-8\lambda)$.

Suppose now that the equation (\ref{st}) has a solution $(\nu_1,\,\nu_2,\,\nu_3)$, where
$\nu_1=\nu_3$. After the change of variables $x=\nu_1$, $y=\nu_1+\nu_2$ the equation (\ref{st}) will written down in the form (\ref{s4}), moreover
 \begin{equation}
\label{sond}
x,\,y\in\mathbb{N},\,\,x<y,\,\,x+y\equiv 1(\Mod 2).
\end{equation}
It is easy to see that condition $x+y\equiv 1(\Mod 2)$ is equivalent to relation
$-4\lambda\in\mathbb{N}$; in addition the number of solutions to equation (\ref{s4}) with condition (\ref{sond}) is equal to $L_2(5-8\lambda).$
\end{proof}

\begin{theorem}
\label{spin7}
Let $G=\Spin(7)$ is supplied by biinvariant Riemannian metric $\nu$ such that $\nu(e)=-k_{ad}$. A number $\lambda < 0$ is an eigenvalue of Laplacian on $(G,\nu)$ in one and one of the following cases (I or II)

I. 1) $-5\lambda\in\mathbb{N}$;

2) natural number $k=35-40\lambda$ is a sum of squares of three mutually different natural numbers.

In addition the number of highest weights $\Lambda$, such that $\lambda(\Lambda)=\lambda$, is equal to $L_3(35-40\lambda)$.

II. 1) $-40\lambda\in\mathbb{N}$, moreover $-40\lambda\equiv 1(\Mod 4)$;

2) natural number $k=\frac{35-40\lambda}{4}$ is a sum of squares of three mutually different natural numbers.

Moreover the number of highest weights $\Lambda$, such that $\lambda(\Lambda)=\lambda$, is equal to  $L_3\left(\frac{35-40\lambda}{4}\right)$.
\end{theorem}

\begin{proof}
Following (\ref{xi4}), consider equation
\begin{equation}
\label{s2}
(2\nu_1+2\nu_2+\nu_3)^2+(2\nu_2+\nu_3)^2+\nu_3^2=35-40\lambda,\quad\nu_1,\,\,\nu_2,\,\,\nu_3\in\mathbb{N}.
\end{equation}
It is clear that if the equation (\ref{s2}) is solvable, then $-40\lambda\in\mathbb{N}$.

Assume at first that the equation (\ref{s2}) has a solution $(\nu_1,\,\nu_2,\,\nu_3)$ such that
$\nu_3$ is an odd number. Introduce the following notation
$$x=\nu_3,\quad y=2\nu_2+\nu_3,\quad z=2\nu_1+2\nu_2+\nu_3.$$
Then the equation (\ref{s2}) will written down as
\begin{equation}
\label{s5}
x^2+y^2+z^2=35-40\lambda,
\end{equation}
moreover the condition (\ref{cond}) is fulfilled.

Then it follows from proof of theorem \ref{qq} that if the equation (\ref{s2}) is solvable, and at least one solution has odd $\nu_3$, then $35-40\lambda\equiv 3(\Mod 8)$, what is equivalent to condition $-5\lambda\in\mathbb{N}$, and the number of solutions, satisfying next to the last condition, is equal to  $L_3(35-40\lambda)$.

Suppose now that the equation (\ref{s2}) has a solution $(\nu_1,\,\nu_2,\,\nu_3)$ such that
$\nu_3$ is an even number. It is clear then that
$35-40\lambda\equiv 0(\Mod 4)$, what is equivalent to relation $-40\lambda\equiv 1(\Mod 4)$.
Introduce the following notation
$$x=\frac{\nu_3}{2},\quad y=\nu_2+\frac{\nu_3}{2},\quad z=\nu_1+\nu_2+\frac{\nu_3}{2}.$$
Then the equation (\ref{s2}) will written as
\begin{equation}
\label{s6}
x^2+y^2+z^2=\frac{35-40\lambda}{4},
\end{equation}
where (\ref{cond1}). Therefore the number of solutions to equation (\ref{s2}) with even $\nu_3$ is equal to $L_3\left(\frac{35-40\lambda}{4}\right)$.
\end{proof}

\begin{example}
1) If $-40\lambda=21,$ then $-40\lambda\equiv 1(\Mod 4),$ $k=(35-40\lambda)/4=14=1^2+2^2+3^2.$ In addition $\lambda=-21/40$ is the least by modulus non-zero number eigenvalue of Laplacian, considered above and the number of highest weights $\Lambda,$ such that $\lambda(\Lambda)=-21/40,$ is equal to $L_3(14)=1.$

2) If $-5\lambda=3,$ then $k=35-40\lambda=59=1^2+3^2+7^2.$ Moreover $\lambda=-3/5,$
$\Lambda=\overline{\omega}_1$ and the number of highest weights $\Lambda',$ such that
$\lambda(\Lambda')=-3/5,$ is equal to $L_3(59)=1.$
\end{example}

\begin{theorem}
Let $G=\SO(7)$ is supplied by biinvariant Riemannian metric $\nu$ such that $\nu(e)=-k_{ad}$. A number $\lambda < 0$ is an eigenvalue of Laplacian on $(G,\nu)$ if and only if the following conditions are fulfilled

1) $-5\lambda\in\mathbb{N}$;

2) natural number $k=35-40\lambda$ is a sum of squares of three mutually different numbers.

In addition the number of highest weights $\Lambda$, such that $\lambda(\Lambda)=\lambda$, is equal to $L_3(35-40\lambda)$.

\end{theorem}

\begin{proof}
It follows from (\ref{xi5}) and proof of theorem \ref{spin7}.
\end{proof}

\begin{theorem}
Let $G=\Sp(3)$ is supplied by biinvariant Riemannian metric $\nu$ such that $\nu(e)=-k_{ad}$. A number $\lambda< 0$ is an eigenvalue of Laplacian on $(G,\nu)$ if and only if the following conditions are satisfied

1) $-16\lambda\in\mathbb{N}$;

2) natural number $k=14-16\lambda$ is a sum of squares of three mutually different numbers.

Moreover the number of highest weights $\Lambda,$ such that $\lambda(\Lambda)=\lambda$, is equal to  $L_3\left(14-16\lambda\right)$.
\end{theorem}

\begin{proof}
Following (\ref{xi6}), consider Diophantine equation
\begin{equation}
\label{s55}
(\nu_1+\nu_2+\nu_3)^2+(\nu_2+\nu_3)^2+\nu_3^2=14-16\gamma\lambda,\quad\nu_1,\,\,\nu_2,\,\,\nu_3\in\mathbb{N}.
\end{equation}
It is clear that if the equation (\ref{s55}) is solvable, then $-16\lambda\in\mathbb{N}$.
Introduce the following notation
\begin{equation}
\label{xyz}
x=\nu_3,\quad y=\nu_2+\nu_3,\quad z=\nu_1+\nu_2+\nu_3.
\end{equation}
Then the equation (\ref{s55}) will written as
\begin{equation}
\label{s66}
x^2+y^2+z^2=14-16\lambda,
\end{equation}
where (\ref{cond1}). Therefore the number of solutions to equation (\ref{s55}) is equal to  $L_3(14-16\lambda)$.
\end{proof}

\begin{example}
1) If $-16\lambda=7,$ then $k=14-16\lambda=21=1^2+2^2+4^2.$ Moreover $\lambda=-7/16$ is the least by modulus non-zero eigenvalue of Laplacian, considered earlier, and the number of highest weights
$\Lambda,$ such that $\lambda(\Lambda)=-7/16,$ is equal to $L_3(21)=1.$

2) If $-16\lambda=12,$ then $k=14-16\lambda=26=1^2+3^2+4^2.$ Moreover $\lambda=-3/4,$
$\Lambda=\overline{\omega}_2$ and the number of highest weights $\Lambda',$ such that
$\lambda(\Lambda')=-3/4,$ is equal to $L_3(26)=1.$
\end{example}

\begin{theorem}
Let $\Sp(3)/C(\Sp(3))$ is supplied by biinvariant Riemannian metric $\nu$ such that
$\nu(e)=-k_{ad}$. A number $\lambda< 0$ is an eigenvalue of Laplacian on $(G,\nu)$ if and only if the following conditions are fulfilled

1) $-8\lambda\in\mathbb{N}$;

2) natural number $k=14-16\lambda$ is a sum of squares of three mutually different natural numbers.

Moreover the number of highest weights $\Lambda,$ such that $\lambda(\Lambda)=\lambda$, is equal to  $L_3\left(14-16\lambda\right)$.
\end{theorem}

\begin{proof}
Following (\ref{xi7}), consider Diophantine equation
\begin{equation}
\label{s50}
(\nu_1+\nu_2+\nu_3)^2+(\nu_2+\nu_3)^2+\nu_3^2=14-16\lambda,\quad\nu_1,\,\nu_2,\,\nu_3\in\mathbb{N},\,\,\nu_1\equiv\nu_3(\Mod 2).
\end{equation}
Under notation (\ref{xyz}) the equation (\ref{s50}) will written in the form (\ref{s66}),
moreover
$$x,\,y,\,z\in\mathbb{N},\,\,x<y<z,\,\,z\equiv x+y(\Mod 2).$$
If $z\equiv x+y(\Mod 2)$, then $z^2\equiv x^2+y^2(\Mod 2)$. Then it follows from (\ref{s66}) that natural number $14-16\lambda$ is even, what is equivalent (for $\lambda<0$) to condition
$-8\lambda\in\mathbb{N}$. Conversely, if $-8\lambda\in\mathbb{N}$, then it is easy to see that it follows from equation (\ref{s66}), where $x,\,y,\,z\in\mathbb{N}$, that $z\equiv x+y(\Mod 2)$. Therefore the number of solutions to equation (\ref{s50}) is equal to $L_3(14-16\lambda)$.
\end{proof}

\section{Laplacian spectrum for groups $\Spin(5)$ and $\SO(5)$}

In work \cite{Svir3} is calculated the Laplace spectrum of Lie groups $\Spin(5)$ and $\SO(5)$, where instead of $\bar{\omega}_1$ was used $\omega=\bar{\omega}_1-\bar{\omega}_2$. But such questions as whether given number is an eigenvalue of Laplacian and how much highest weights correspond to one and the same eigenvalue of Laplacian, remained open. In this section we shall calculate Laplacian spectrum without usage of $\omega$ and will give an answer to these questions.

To Lie groups $\Spin(5)$ и $\SO(5)$ corresponds Lie algebra $\mathfrak{so}(5)$ with root system   $B_2$. We apply Table~II. Simple roots are
$\{\alpha_1=\varepsilon_1-\varepsilon_2,\,\,\alpha_2\}$.
Positive roots are $\varepsilon_i$, $i=1,2$; $\varepsilon_1\pm\varepsilon_2$.
The sum of positive roots is equal to $2\beta=3\varepsilon_1+\varepsilon_2$, whence
\begin{equation}
\label{k2}
\beta=\frac{3\varepsilon_1+\varepsilon_2}{2},\quad \tilde{\alpha}+\beta=\frac{5\varepsilon_1+3\varepsilon_2}{2}.
\end{equation}

1) $b=\langle\tilde{\alpha}+\beta,\tilde{\alpha}+\beta\rangle-\langle\beta,\beta\rangle=\frac{34}{4}-\frac{10}{4}=6.$

2) $\left(\cdot,\cdot\right)=\frac{1}{6}\langle\cdot,\cdot\rangle.$

3) Fundamental weights have the form
$\{\bar{\omega}_1=\varepsilon_1,\,\,
\bar{\omega}_2=\frac{1}{2}(\varepsilon_1+\varepsilon_2)\}.$

4) Set $\Lambda=\Lambda_1\bar{\omega}_1+\Lambda_2\bar{\omega}_2\in\Lambda^{+}(\mathfrak{so}(5)),$ where $\Lambda_i\in\mathbb{Z}$, $\Lambda_i\geq 0$, $i=1,2$. Then $\Lambda+\beta=\nu_1\bar{\omega}_1+\nu_2\bar{\omega}_2$, where $\nu_i=\Lambda_i+1$, $\nu_i\in\mathbb{N}$, $i=1,2.$

By formula (\ref{lam}), eigenvalue $\lambda(\Lambda),$ corresponding to highest weight $\Lambda$,  is equal to
\begin{equation}
\label{x1}
\lambda(\Lambda)=-\frac{1}{6\gamma}[\langle\Lambda+\beta,\Lambda+\beta\rangle-\langle\beta,\beta\rangle]=-\frac{1}{24\gamma}\left((2\nu_1+\nu_2)^2+\nu_2^2-10\right).
\end{equation}

By formula  (\ref{dim}), dimension $d(\Lambda+\beta)$ of representation
$\rho(\Lambda),$ corresponding to highest weight $\Lambda$, is equal to
$$d(\Lambda+\beta)=\frac{(\Lambda+\beta,\varepsilon_1)}{(\beta,\varepsilon_1)}\cdot\frac{(\Lambda+\beta,\varepsilon_2)}{(\beta,\varepsilon_2)}\cdot
\frac{(\Lambda+\beta,\varepsilon_1-\varepsilon_2)}{(\beta,\varepsilon_1-\varepsilon_2)}\cdot
\frac{(\Lambda+\beta,\varepsilon_1+\varepsilon_2)}{(\beta,\varepsilon_1+\varepsilon_2)}.$$
We have
$$\frac{(\Lambda+\beta,\varepsilon_1)}{(\beta,\varepsilon_1)}=
\frac{2\nu_1+\nu_2}{3},\quad
\frac{(\Lambda+\beta,\varepsilon_2)}{(\beta,\varepsilon_2)}=\nu_2,$$
$$\frac{(\Lambda+\beta,\varepsilon_1+\varepsilon_2)}{(\beta,\varepsilon_1+\varepsilon_2)}=
\frac{\nu_1+\nu_2}{2},\quad
\frac{(\Lambda+\beta,\varepsilon_1-\varepsilon_2)}{(\beta,\varepsilon_1-\varepsilon_2)}=
\nu_1.$$
Consequently
\begin{equation}
\label{x2}
d(\Lambda+\beta)=\frac{1}{6}\nu_1\nu_2(\nu_1+\nu_2)(2\nu_1+\nu_2).
\end{equation}

5) Simple roots and fundamental weights of Lie algebra $\mathfrak{so}(5),$ indicated above, define respective lattices $\Lambda_0(B_2)$ and $\Lambda_1(B_2)$ in the following way
\begin{equation}
\label{x3}
\Lambda_0(B_2)=\left\{\Psi_1\alpha_1+\Psi_2\alpha_2\,\mid\,\Psi_1,\Psi_2\in\mathbb{Z}\right\},\quad
\Lambda_1(B_2)=\left\{\Lambda_1\bar{\omega}_1+\Lambda_2\bar{\omega}_2\,\mid\,\Lambda_1,\Lambda_2\in\mathbb{Z}\right\}.
\end{equation}

After expressing roots via fundamental weights
$$\alpha_1=2\bar{\omega}_1-2\bar{\omega}_2,\quad \alpha_2=-\bar{\omega}_1+2\bar{\omega}_2$$
and changing variables
$$\left\{
\begin{array}{l}
\Psi_1=\Omega_1+\Omega_2, \\
\Psi_2=\Omega_1+2\Omega_2
\end{array}\right.\quad\Leftrightarrow\quad
\left\{\begin{array}{l}
\Omega_1=2\Psi_1-\Psi_2, \\
\Omega_2=-\Psi_1+\Psi_2
\end{array}\right.$$
the lattice $\Lambda_0(B_2)$ will have the following form
\begin{equation}
\label{x4}
\Lambda_0(B_2)=\{\Omega_1\bar{\omega}_1+2\Omega_2\bar{\omega}_2\,\,\mid\,\Omega_1,\Omega_2\in\mathbb{Z}\}.
\end{equation}

It follows from formulas (\ref{x3}) and (\ref{x4}) that $\Lambda_0(B_2)\subset\Lambda_1(B_2)$ and
$\Lambda_1(B_2)/\Lambda_0(B_2)\cong\mathbb{Z}_2$, i.e. it has prime order. This and corollary \ref{Col:LCalc} imply that there is no other lattice. The lattice $\Lambda_1(B_2)$ corresponds to compact simply connected Lie group with Lie algebra $\mathfrak{so}(5)$, i.e. to Lie group
$\Spin(5)$, while the lattice $\Lambda_0(B_2)$ corresponds to compact non simply connected Lie group with Lie algebra $\mathfrak{so}(5)$, i.e. to Lie group $\SO(5)$.

a) Let us present formulas, giving Laplacian spectrum of the Lie group $\Spin(5)$.

6a) By formula (\ref{x3}), the set of highest weights of the Lie group $\Spin(5)$ is equal to
$$\Lambda^{+}(\Spin(5))=\left\{\Lambda_1\bar{\omega}_1+\Lambda_2\bar{\omega}_2\,\mid\,\Lambda_i\in\mathbb{Z},\,\,\Lambda_i\geq 0,\,\,i=1,2\right\}.$$

7a) Set $\Lambda=\Lambda_1\bar{\omega}_1+\Lambda_2\bar{\omega}_2\in\Lambda^{+}(\Spin(5))$, then
$\Lambda+\beta=\nu_1\bar{\omega}_1+\nu_2\bar{\omega}_2$, where $\nu_i=\Lambda_i+1$, $\nu_i\in\mathbb{N}$, $i=1,2.$

8a) Applying formula (\ref{abc}) and results of preceding point, we get the following multiplicity of eigenvalue $\lambda(\Lambda)$
$$\sigma(\Lambda)=\frac{1}{36}\sum_{\Omega_1}
\nu_1^2\nu_2^2(\nu_1+\nu_2)^2(2\nu_1+\nu_2)^2,$$
where
\begin{equation}
\label{Om1}
\Omega_1=\{(\nu_1,\nu_2)\in\mathbb{N}^2\,\mid\,(2\nu_1+\nu_2)^2+\nu_2^2=10-24\gamma\lambda\}.
\end{equation}

The least by modulus non-zero eigenvalue of Laplacian is equal to $-\frac{5}{12\gamma}$ and corresponds to irreducible complex representation of the Lie group $\Spin(5)$ with highest weight
$\bar{\omega}_2$. Dimension of this representation is equal to $4$. Consequently the multiplicity of the eigenvalue $-\frac{5}{12\gamma}$ is equal to $4^2=16$.

b) Present formulas, giving Laplacian spectrum of the Lie group $\SO(5)$.

6b) By formula (\ref{x4}), the set of highest weights of Lie group $\SO(5)$ is equal to
$$\Lambda^{+}(\SO(5))=\{\Omega_1\bar{\omega}_1+2\Omega_2\bar{\omega}_2\,\,\mid\,\Omega_i\in\mathbb{Z},\,\,\Omega_i\geq 0,\,\,i=1,2\}.$$

7b) Set $\Lambda=\Omega_1\bar{\omega}_1+2\Omega_2\bar{\omega}_2\in\Lambda^{+}(\SO(5))$, then by p.~4) eigenvalue $\lambda(\Lambda)$ and dimension $d(\Lambda+\beta)$ are calculated respectively
by formulas  (\ref{x1}) and (\ref{x2}), where
$\nu_1=\Omega_1+1$, $\nu_2=2\Omega_2+1$, $\nu_i\in\mathbb{N}$, $i=1,2,$ $\nu_2\equiv 1(\Mod 2)$.

8b) Applying formula (\ref{abc}) and results of preceding point, we get the following multiplicity of the eigenvalue $\lambda(\Lambda)$
$$\sigma(\Lambda)=\frac{1}{36}\sum_{\Omega_2}
\nu_1^2\nu_2^2(\nu_1+\nu_2)^2(2\nu_1+\nu_2)^2,$$
where
\begin{equation}
\label{Om2}
\Omega_2=\{(\nu_1,\nu_2)\in\mathbb{N}^2\,\mid\,(2\nu_1+\nu_2)^2+\nu_2^2=10-24\gamma\lambda,\,\,\nu_2\equiv 1(\Mod 2)\}.
\end{equation}

The least by modulus non-zero eigenvalue of Laplacian is equal to $-\frac{2}{3\gamma}$ and corresponds to irreducible complex representation of the Lie group $\SO(5)$ with highest weight
$\bar{\omega}_1$. Dimension of this representation is equal to $5$. Consequently the multiplicity of the eigenvalue $-\frac{2}{3\gamma}$ is equal to $5^2=25$.

\begin{theorem}
\label{spin5}
Let $G=\Spin(5)$ is supplied by biinvariant metric $\nu$ such that $\nu(e)=-k_{ad}$. A number
$\lambda < 0$ is an eigenvalue of Laplacian on $(G,\nu)$ for one and only one of the following cases (I or II)

I. 1) $-3\lambda\in\mathbb{N}$;

2) natural number $k=10-24\lambda$ is a sum of squares of two mutually different natural numbers.

Moreover the set of highest weights $\Lambda,$ such that $\lambda(\Lambda)=\lambda$, is equal to $L_2(10-24\lambda)$.

II. 1) $-12\lambda\in\mathbb{N}$, $-12\lambda\geq 5$, $-12\lambda\equiv 1(\Mod 2)$;

2) natural number $k=\frac{5-12\lambda}{2}$ is a sum of squares of two different natural numbers.

In addition the number of highest weights $\Lambda,$ such that $\lambda(\Lambda)=\lambda$, is equal to $L_2\left(\frac{5-12\lambda}{2}\right)$.
\end{theorem}

\begin{proof}
Following (\ref{Om1}), let us consider Diophantine equation
\begin{equation}
\label{sss2}
(2\nu_1+\nu_2)^2+\nu_2^2=10-24\lambda,\quad\nu_1,\,\,\nu_2\in\mathbb{N}.
\end{equation}
It is clear that if the equation (\ref{sss2}) is solvable, then $-24\lambda\in\mathbb{N}$.

Assume at first that the equation (\ref{sss2}) has a solution $(\nu_1,\,\nu_2)$ such that $\nu_2$ is an odd number. After introducing notation $x=\nu_2,$ $y=2\nu_1+\nu_2$ the equation (\ref{sss2}) will written as
\begin{equation}
\label{sss5}
x^2+y^2=10-24\lambda,
\end{equation}
where
\begin{equation}
\label{condit}
x,\,y\in\mathbb{N},\,\,x\equiv 1(\Mod 2),\,\,y\equiv 1(\Mod 2),\,\,x<y.
\end{equation}

On the ground of proposition \ref{AAA}, if the equation (\ref{sss5}) is solvable under condition (\ref{condit}), then
$10-24\lambda\equiv 2(\Mod 8)$, what is equivalent to condition $-3\lambda\in\mathbb{Z}$, and the number of solutions to the equation (\ref{sss5})  is equal to $L_2(10-24\lambda)$.

Let suppose now that the equation (\ref{sss2}) has a solution $(\nu_1,\,\nu_2)$ such that
$\nu_2$ is even number. It is clear that then $10-24\lambda\equiv 0(\Mod 4)$, what is equivalent to relation $-12\lambda\equiv 1(\Mod 2)$. In notation $x=\frac{\nu_2}{2}$, $y=\nu_1+\frac{\nu_2}{2}$ the equation (\ref{sss2}) takes the form
\begin{equation}
\label{sss6}
x^2+y^2=\frac{5-12\lambda}{2},\quad x,\,y\in\mathbb{N},\,\,x<y.
\end{equation}
where (\ref{condit}). Therefore the number of solutions to equation (\ref{sss2}) with even $\nu_2$ is equal to $L_2\left(\frac{5-12\lambda}{2}\right)$.
\end{proof}

\begin{theorem}
Let $G=\SO(5)$ is supplied by biinvariant Riemannian metric $\nu$ such that $\nu(e)=-k_{ad}$. A number $\lambda < 0$ is an eigenvalue of Laplacian on $(G,\nu)$ if and only if the following conditions are fulfilled

1) $-3\lambda\in\mathbb{N}$;

2) natural number $k=10-24\lambda$ is a sum of squares of two different natural numbers.

Moreover the number of highest weights $\Lambda,$ such that $\lambda(\Lambda)=\lambda$, is equal to $L_2(10-24\lambda)$.
\end{theorem}

\begin{proof}
It follows from (\ref{Om2}) and proof of theorem \ref{spin5}.

\end{proof}

\end{document}